\documentclass[11pt]{amsart}
\usepackage{amsmath,amssymb,amsthm}
\usepackage[latin1]{inputenc}
\usepackage{version, tabularx, multicol}
\usepackage{graphicx,float}
\usepackage{stmaryrd}
\usepackage{color,latexsym,amsfonts}

\newcommand{\reff}[1]{(\ref{#1})}

\theoremstyle{plain}
\newtheorem{theo}{Theorem}[section]
\newtheorem{theo*}{Theorem}
\newtheorem{cor}[theo]{Corollary}
\newtheorem{lem}[theo]{Lemma}
\newtheorem{defi}[theo]{Definition}
\theoremstyle{remark}
\newtheorem{rem}[theo]{Remark}

\newcommand{\co}{\mathcal O}

\newcommand{\ca}{{\mathcal A}}
\newcommand{\cb}{{\mathcal B}}

\newcommand{\cf}{{\mathcal F}}

\newcommand{\ci}{{\mathcal I}}
\newcommand{\cj}{{\mathcal J}}

\newcommand{\cl}{{\mathcal L}}

\newcommand{\cm}{{\mathcal M}}

\newcommand{\cq}{{\mathcal Q}}
\newcommand{\cs}{{\mathcal S}}
\newcommand{\Tau}{{\mathcal T}}

\newcommand{\A}{{\mathbb A}}
\newcommand{\C}{{\mathbb C}}

\newcommand{\E}{{\mathbb E}}

\newcommand{\F}{{\mathbb F}}

\newcommand{\N}{{\mathbb N}}
\renewcommand{\P}{{\mathbb P}}

\newcommand{\R}{{\mathbb R}}

\newcommand{\T}{{\mathbb T}}
\newcommand{\Z}{{\mathbb Z}}

\newcommand{\rP}{{\rm P}}
\newcommand{\rE}{{\rm E}}

\newcommand{\bt}{{\mathbf t}}
\newcommand{\bff}{{\mathbf f}}

\newcommand{\ind}{{\bf 1}}

\newcommand{\fa}{\mathrm{Pa}}
\newcommand{\dom}{{\rm dom}}

\newcommand{\ri}{{\rm ri}\;}
\newcommand{\aff}{{\rm aff}\;}
\newcommand{\cv}{{\rm cv}\;}
\newcommand{\clo}{{\rm cl}\;}
\newcommand{\inter}{{\rm int}\;}
\newcommand{\Span}{{\rm span}\;}

\newcommand{\supp}{{\rm supp}\;}

\newcommand{\Card}{{\rm Card}\;}

\newcommand{\norm}[1]{\mathop{\parallel #1  \parallel}\nolimits}
\newcommand{\val}[1]{\mathop{\left| #1 \right|}\nolimits}
\newcommand{\inv}[1]{\mathop{\frac{1}{ #1}}\nolimits}
\newcommand{\expp}[1]{\mathop {\mathrm{e}^{ #1}}}

\newcommand{\Var}{{\rm Var}\;}
\newcommand{\Cov}{{\rm Cov}\;}

\title{Critical multi-type Galton-Watson trees conditioned to be large}

\date{\today}

\author{Romain Abraham}
\address{Romain Abraham,
 Laboratoire MAPMO, CNRS, UMR 7349,
 F\'ed\'eration Denis Poisson, FR 2964,
  Universit\'{e} d'Orl\'{e}ns,
 B.P. 6759,
 45067 Orl\'{e}ans cedex 2,
France}
\email{romain.abraham@univ-orleans.fr}

\author{Jean-Fran\c{c}ois Delmas}
\address{Jean-Fran\c{c}ois Delmas,
 Universit\'{e} Paris-Est, CERMICS (ENPC), F-77455 Marne La Vall\'{e}e, France}
\email{delmas@cermics.enpc.fr} 

\author{Hongsong Guo}
\address{Hongsong Guo,
 Universit\'{e} Paris-Est, CERMICS (ENPC), F-77455 Marne La Vall\'{e}e,
 France; Department of Mathematics,
China University of Mining and Technology,
Beijing 100083, P. R. China}
\email{hsguo@mail.bnu.edu.cn}

\begin{document}

\subjclass[2010]{60J80; 60B10}

\keywords{Galton-Watson; random tree; local-limit; strong ratio theorem;
branching process}

\begin{abstract}
  Under  minimal  condition, we  prove  the  local convergence of  a  critical
  multi-type  Galton-Watson tree  conditioned  on having  a large  total
  progeny by  types towards  a multi-type Kesten's  tree.  We  obtain the
  result  by  generalizing  Neveu's   strong  ratio  limit  theorem  for
  aperiodic random walks on $\Z^d$.
\end{abstract}

\maketitle

\section{Introduction}

In  \cite{K86}, Kesten  shows  that the  local limit  of  a critical  or
subcritical Galton-Watson (GW) tree conditioned on having a large height
is an infinite  GW tree (in fact  a multi-type GW tree  with one special
individual per generation) with a  unique infinite spine, which we shall
call {\it  Kesten's ~tree} in the  present paper. In Abraham  and Delmas
\cite{AD14a} a  sufficient and necessary  condition is given for  a wide
class of  conditionings for a  critical GW  tree to converge  locally to
Kesten's tree  under minimal  hypotheses on the  offspring distribution.
Notice  that condensation  may  arise when  considering sub-critical  GW
trees,   see   Janson   \cite{j:sgtcgwrac},   Jonnson   and   Stefansson
\cite{js:cnt}, He  \cite{h:cgwtmod} or  Abraham and  Delmas \cite{AD14b}
for results in this direction. When  scaling limits of multi-type GW tree
are considered, one obtains as a limit a continuous GW tree, see Miermont
\cite{M08}  or  Gorostiza  and   Lopez-Mimbela  \cite{glm90}  (when  the
probability to  give birth to different  types goes down to  0). In this
latter  case see  Delmas  and  Hénard \cite{dh}  for  the  limit on  the
conditioned random tree to have a large height.
\\

In  the  multi-type case,  P{\'e}nisson  \cite{P14}  has proved  that  a
critical $d$-types  GW process  conditioned on the  total progeny  to be
large with a given asymptotic  proportion of types converges locally to
a multi-type GW process (with a special individual per generation) under
the  condition  that  the  branching process  admits  moments  of  order
$d+1$. Stephenson \cite{S14} gave,  under an exponential moments condition,
the local convergence  of a multi-type GW tree, conditioned  on a linear
combination of  population sizes of each  type to be large,  towards the
multi-type Kesten's tree introduced by  Kurtz, Lyons, Pemantle and Peres
\cite{KLPP97}.  The aim  of this paper is to give  minimal hypotheses to
ensure  the  local   convergence  of  a  critical   multi-type  GW  tree
conditioned  on the  total progeny  to be  large towards  the associated
multi-type Kesten's tree, see Theorem \ref{dTheorem}. When the offspring
distribution is  aperiodic, the minimal  hypotheses is the  existence of
the  mean matrix  which is  assumed  to be  primitive.  Furthermore,  we
exactly condition on  the asymptotic proportion of types  for the total
progeny of the GW tree to  be given by the (normalized) left eigenvector
associated with the Perron-Frobenius eigenvalue of the mean matrix.
\\

If the asymptotic proportion of types  is not equal to the (normalized)
left eigenvector associated with  the Perron-Frobenius eigenvalue of the
mean matrix, then under an exponential  moments condition for the offspring
distribution, it is possible to get  a Kesten's tree as local limit, see
\cite{P14}.   However, without  an exponential  moments  condition for  the
offspring distribution no results are known, and results in \cite{AD14b}
for the mono-type  case suggests a condensation phenomenon  (at least in
the  sub-critical   case).   Conditioning  large  multi-type   (or  even
mono-type) continuous GW  tree to have a large population  in the spirit
of \cite{dh} is also an open question.
\\

The proof of  Theorem \ref{dTheorem} relies on two  arguments. The first
one is  a generalization  of the Dwass formula  for multi-type  GW processes
given  by  Chaumont  and  Liu  \cite{CL13}  which  encodes  critical  or
sub-critical  $d$-multi-type  GW  forests  using  $d$  random  walks  of
dimension $d$.   The second one is  the strong ratio theorem  for random
walks  in $\Z^d$,  see Theorem  \ref{theo:neveuth}, which  generalizes a
result by  Neveu \cite{N63} in dimension  one.  The proof of  the strong
ratio theorem relies  on a uniform version of  the $d$-dimensional local
theorem  of  Gnedenko  \cite{G48},  see  also  Gnedenko  and  Kolmogorov
\cite{GK54}  (for  the  sum  of independent  random  variables),  Rvaceva
\cite{R61} (for the sum of  $d$-dimensional i.i.d.  random variables) or
Stone \cite{S66} (for the sum  of $d$-dimensional i.i.d.  lattice or non
lattice random variables), which is given in Section \ref{sec:Gnedenko},
and  properties  of the  Legendre-Laplace  transform of  a  probability
distribution.  As we were unable to  find those latter properties in the
literature,   we  give   them  in   a  general   framework  in   Section
\ref{sec:prel}, as we believe they
might be interesting by themselves.\\

The paper is organized as follows. We present in Section \ref{sec:multi}
the  topology on the set of  the multi-type  trees and  a sufficient  and necessary
condition  for the  local convergence  of random  multi-type trees,  see
Corollary  \ref{cor:cv}, the  definition of  a multi-type GW  tree with  a
given  offspring  distribution and  the  aperiodicity  condition on  the
offspring distribution,  see Definition  \ref{defi:aperiodic-p}. Section
\ref{sec:main} is  devoted to  the main result,  Theorem \ref{dTheorem},
and its  proof. The  Appendix collects results  on the  Legendre-Laplace
transform in  a general framework in  Section \ref{sec:prel}, Gnedenko's
$d$-dimensional  local theorem  in Section  \ref{sec:Gnedenko}, and  the
strong ratio limit  theorem for $d$-dimensional random  walks in Section
\ref{sec:rlrw}.

\section{Multi-type trees}
\label{sec:multi}
\subsection{General notations}
We denote  by $\mathbb{N}=\{0, 1,  2, \ldots\}$ the set  of non-negative
integers and by  $\mathbb{N}^*=\{1, 2, \ldots\}$ the  set of positive
integers. For $d\in \N^*$, we set $[d]=\{1,\ldots, d\}$.

Let $d\geq1$.  We  say $x=(x_i, i\in [d])\in\R^d$ is a  column vector in
$\R^d$.   We  write $\ind=(1,\ldots,1)\in\R^d$,  $0=(0,\ldots,0)\in\R^d$
and denote by $\mathbf{e}_i$ the vector  such that the $i$-th element is
1 and others are 0.  For vectors $x=(x_i, i\in [d])\in\R^d$ and $y=(y_i,
i\in [d])\in\R^d$, we  denote by $\langle x, y\rangle$  the usual scalar
product of $x$ and $y$,  by $x^y$ the product $\prod_{i=1}^d x_i^{y_i}$,
by  $|x|=\sum_{i=1}^d |x_i|$  and $\norm{x}=\sqrt{\langle  x,x \rangle}$
the $\ell^1$ and $\ell^2$ norms of $x$,  and we write $x\leq y$ (resp.  $x<y$)
if $x_i\leq y_i$ (resp. $x_i<y_i$) for all $i\in [d]$.

For  any nonempty  set  $A\subset  \R^d$, we define $\Span  A$ as the  linear
sub-space generated  by $A$  (that is  $\Span A=\{\sum_{i=1}^n\alpha_i
y_i; \, \alpha_i\in \R, y_i\in A, i\in [n], n\in\N^*\}$) and for
$x\in\R^d$, we denote $x+A=\{x+y; y\in A\}$. For $A$ and $B$ nonempty
subsets of $\R^d$, we denote $A-B=\{x-y; x\in A, y\in B\}$. 

For a random variable $X$ and an event $A$, we write $\E[X;\,
A]$ for $\E\left[X\ind_A\right]$. 

\subsection{Notations for marked trees}
\label{sec:not-mt}
Let  $d\in  \N^*$.  Denote  by  $[d]$  the set  of  types  or marks,  by
$\widehat{\mathcal {U}}=\bigcup_{n\geq0}(\mathbb{N}^*)^n$  the set of
finite   sequences   of   positive    integers   with   the   convention
$(\mathbb{N}^*)^0=\{\widehat{\emptyset}\}$   and    by   $   \mathcal
{U}=\bigcup_{n\geq0}\Big((\mathbb{N}^*)^n\times[d]\Big) $  the set of
finite  sequences  of positive  integers  with  a  type.  For  a  marked
individual  $u\in\mathcal{U}$, we  write $u=(\hat{u},  \mathcal {M}(u))$
with $\hat{u}\in  \widehat{\mathcal {U}}$  the individual  and $\mathcal
{M}(u)\in[d]$ its  type or mark.   Let $|u|=|\hat{u}|$ be the  length or
height  of $u$  defined  as  the integer  $n$  such that  $\hat{u}=(u_1,
\ldots, u_n)\in(\mathbb{N}^*)^n$.  If $\hat{u}$ and $\hat{v}$ are two
sequences in $\widehat{\mathcal {U}}$, we denote by $\hat{u}\hat{v}$ the
concatenation   of  the   two  sequences,   with  the   convention  that
$\hat{u}\hat{v}=\hat{u}$     if    $\hat{v}=\widehat{\emptyset}$     and
$\hat{u}\hat{v}=\hat{v}$  if   $\hat{u}=\widehat{\emptyset}$.   For  $u,
v\in\mathcal{U}$, we  denote by  $uv$ the concatenation  of $u$  and $v$
such           that          $\widehat{uv}=\hat{u}\hat{v}$           and
$\mathcal{M}(uv)=\mathcal{M}(v)$              if             $|v|\geq1$;
$\mathcal{M}(uv)=\mathcal{M}(u)$ if $|v|=0$.   Let $u, v\in\mathcal{U}$.
We say that $v$ (resp. $\hat v$)  is an ancestor of $u$ (resp. $\hat u$)
and write $v\preccurlyeq u$ (resp. $\hat v\preccurlyeq \hat u$) if there
exists  $w\in{\mathcal{U}}$ such  that $u=vw$  (resp. $\hat  w \in  \hat
{\mathcal{U}}$ such
that $\hat u=\hat v \hat w$).\\

A  tree $\hat \bt$ is a subset of $\hat{\mathcal{U}}$ such that:
\begin{itemize}
  \item $\widehat{\emptyset}\in \hat{\bt}$.
  \item   If  $\hat{u}\in   \hat{\bt}$,  then   $\{\hat{v};  \hat v\preccurlyeq
    \hat u\}\subset \hat{\bt}$.
  \item For every $\hat u\in  \hat{\bt}$, there exists $k_{\hat u} [\hat
    \bt]\in\mathbb{N}$  such that,  for every  positive integer  $\ell$,
    $\hat{u}\ell\in{\hat \bt}$ iff $1\leq  \ell \leq k_{\hat u}[\hat
    \bt]$.
\end{itemize}

A marked tree $\bt$ is a subset of $\mathcal{U}$ such that:
\begin{itemize}
  \item[(a)]  The set $\hat{\bt}=\{\hat{u}; u\in \bt\}$ of (unmarked)
    individuals of $\bt$ is a tree. 
\item[(b)]  There is  only one  type per  individual: for  $u, v\in  \bt$,
  $\hat{u}=\hat{v}$  implies  $\mathcal{M}(u)=\mathcal{M}(v)$  and  thus
  $u=v$.
\end{itemize}

Thanks to (b),  the number of offsprings of the  marked individual $u\in
\bt$,  $k_u[\bt]$, corresponds  to  $k_{\hat u}  [\hat  \bt]$.  In  what
follows we will deal only with marked trees and simply call them trees.

Denote              by              $\emptyset_\bt=(\widehat{\emptyset},
\mathcal{M}(\emptyset_\bt))\in\mathcal{U}$  the root  of the  tree $\bt$
and write  $\emptyset$ instead  of $\emptyset_\bt$  when the  context is
clear.  The parent of $v\in \bt\setminus \emptyset_\bt$ in $\bt$, denoted by
$\fa_v(\bt)$ is the only $u \in \bt$ such that $|u|=|v|-1$ and $u\preccurlyeq v$. 
The set of the children of $u\in \bt$ is
\[
C_u(\bt)=\{v\in  \bt, \; \fa_v(\bt)=u\}.
\]
Notice   that   $k_u[\bt]=\Card(C_u(\bt))$   for   $u\in\bt$.   We   set
$k_u(\bt)=(k_u^{(i)}[\bt], i\in [d])$, where for $i\in [d]$
\[
k_u^{(i)}[\bt]=\Card(\{v\in C_u(\bt); \, \mathcal{M}(v)=i\})
\]
is  the  number  of offsprings  of  type  $i$  of  $u\in \bt$.  We  have
$\sum_{i\in [d]}  k_u^{(i)}[\bt]= k_u[\bt]$.   The vertex $u\in  \bt$ is
called a leaf  if $k_u[\bt]=0$ and let  $\mathcal {L}_0(\bt)=\{u\in \bt,
k_u[\bt]=0\}$ be the set of leaves of $\bt$.

We   denote   by   $\mathbb{T}$   the  set   of   marked   trees.    For
$\bt\in\mathbb{T}$,  we   define  $|\bt|=(|\bt^{(i)}|,   i\in[d])$  with
$|\bt^{(i)}|=\Card(\{u\in  \bt,   \mathcal{M}(u)=i\})$  the   number  of
individuals   in    $\bt$   of   type    $i$.    Let   us    denote   by
$\mathbb{T}_0=\{\bt\in\mathbb{T}:  \Card(\bt)<\infty\}$  the  subset  of
finite   trees.     We   say    that   a    sequence   $\mathbf{v}=(v_n,
n\in\N)\subset\mathcal{U}$  is an  infinite spine  if $v_n  \preccurlyeq
v_{n+1} $ and $|v_n|=n$ for all  $n\in \N$.  We denote by $\mathbb{T}_1$
the subset  of trees which  have one and  only one infinite  spine.  For
$\bt\in\mathbb{T}_1$, denote  by $\mathbf{v}_\bt$ the infinite  spine of
the tree  $\bt$.  Let  $\mathbb{T}_1'$ be  the subset  of $\mathbb{T}_1$
such that the infinite spine features each type infinitely many times:
\[
\mathbb{T}'_1=\{\bt\in\mathbb{T}_1;\;\forall i\in[d],\, \Card(\{v\in
\mathbf{v}_\bt; \, \mathcal{M}(v)=i\})=\infty \}. 
\]

The height of a tree $\bt$ is defined by $H(\bt)=\sup\{|u|, u\in \bt\}$.
For $h\in\mathbb{N}$, we denote by $\mathbb{T}^{(h)}=\{\bt\in\mathbb{T};
H(\bt)\leq h\}$  the subset  of marked  trees with  height less  than or
equal to $h$.

 \subsection{Convergence determining class}
 \setcounter{equation}{0}

 For $h\in\mathbb{N}$, the restriction  function $r_h$ from $\mathbb{T}$
 to  $\mathbb{T}$  is defined  by  $r_h(\bt)=\{u\in  \bt, |u|\leq  h\}$.
 We endow  the set  $\mathbb{T}$ with  the ultra-metric  distance $  d(\bt,
 \bt')=2^{-\max\{h\in\mathbb{N},   r_h(\bt)=r_h(\bt')\}}$.   The   Borel
 $\sigma$-field  associated  with  the  distance  $d$  is  the  smallest
 $\sigma$-field  containing the  singletons for  which the  restrictions
 $(r_h,  h\in\N)$ are  measurable. With  this distance,  the restriction
 functions   are  continuous.    Since   $\mathbb{T}_0$   is  dense   in
 $\mathbb{T}$   and  $(\mathbb{T},   d)$  is   complete,  we   get  that
 $(\mathbb{T}, d)$ is a Polish metric space.

 Let $\bt, \bt'\in \mathbb{T}$  and $x\in\mathcal{L}_0(\bt)$.  If the type
 of  the  root  of  $\bt'$   is  $\mathcal{M}(x)$,  we  denote  by
\[
 \bt\otimes(\bt',x)=\bt\cup\{xv,v\in \bt'\} 
\]
the tree obtained by grafting the tree  $\bt'$ on the leaf $x$ of the tree
$\bt$;  otherwise,  let  $\bt\otimes(\bt',x)=\bt$.    Then  we  consider 
\[
\mathbb{T}(\bt,x)=\{\bt\otimes(\bt',x),  \bt'\in\mathbb{T}\} 
\]
the set of trees  obtained by grafting a tree on the  leaf $x$ of $\bt$.
For $\bt\in \T_0$, it is easy to see that $\mathbb{T}(\bt,x)$ is closed and also open.

Set $\cf=\{\mathbb{T}(\bt,x);\, \bt\in\mathbb{T}_0, x\in\mathcal
{L}_0(\bt)                 \text{                  and                 }
\mathcal{M}(\emptyset_\bt)=\mathcal{M}(x)\}\cup\{       \{\bt\};      \,
\bt\in\mathbb{T}_0\}$.    Following   the   proof  of   Lemma   2.1   in
\cite{AD14a}, it is easy to get the following result.
\begin{lem}\label{conv-determ}
The family $\cf$ is a convergence determining class on
$\mathbb{T}_0\cup\mathbb{T}'_1$. 
\end{lem}

We deduce the following corollary. 
\begin{cor}
   \label{cor:cv}
Let $(T_n, n\in \N^*)$ and $T$ be random variables taking values in
$\T_0\bigcup  \T'_1$. Then the
sequence $(T_n, n\in \N^*)$ converges in distribution towards  $T$ if and
only if we have for all $\bt\in \T_0$ $\lim_{n\rightarrow+\infty }
\P(T_n=\bt)=\P(T=\bt)$ and for all  $x\in\mathcal
{L}_0(\bt) $ such that $\mathcal{M}(\emptyset_\bt)=\mathcal{M}(x)$:
\[
\lim_{n\rightarrow+\infty } \P(T_n\in \T(\bt, x))=\P(T\in \T(\bt, x)).
\]
\end{cor}

\subsection{Aperiodic distribution}

Let  us consider  a probability  distribution $F=(F(x),  x\in \Z^d)$  on
$\Z^d$.  In order to avoid degenerate cases, we assume that there exists
$x_0\in\Z^d$ such  that:
\begin{equation}
   \label{nondegenerate}
 0<F(x_0)<1.  
 \end{equation} 
Denote by $\supp(F)=\{x\in\Z^d, F(x)>0\}$ the support set of $F$ and by
$R_0$ the smallest subgroup  of $\Z^d$ which contains the set
$\supp(F)-\supp(F)$. 

\begin{defi}\label{defi:aperiodic}
A distribution $F$ on $\Z^d$ is called aperiodic if $R_0=\Z^d$.
\end{defi}

For $x\in \Z^d$, let $G_x$  be the  smallest subgroup of  $\Z^d$ that
  contains $  -x+\supp(F)$. According to the next lemma, an aperiodic
  distribution  is called strongly aperiodic in 
\cite[p.42]{S01}. 

\begin{lem}
   \label{lem:ape}
 If $x\in \supp(F)$, then $G_x=R_0$. The distribution
 $F$ is aperiodic if and only if $G_x=\Z^d$ for some $x\in \supp(F)$ or
 equivalently if and only if $G_x=\Z^d$ for all $x\in \Z^d$.
\end{lem}

\begin{proof}
 Let $x\in \Z^d$. 
Let $z\in  R_0$. There exists  $n, n'\in
  \N$  and  $x_i, x'_i, y_i,y'_i\in  \supp(F)$   for  all $i\in  \N  ^*$  such  that
  $\sum_{i=1}^n  (y_i-x_i)  -   \sum_{i=1}^{n'}  (y'_i-x'_i)  =z$. 
This  implies  that $\sum_{i=1}^n  (y_i-x) + \sum_{i=1}^{n'}  (x'_i-x) 
- \sum_{i=1}^n  (x_i-x) - \sum_{i=1}^{n'}  (y'_i-x)  =z$ and thus
$z\in G_{x}$. This gives $R_0\subset G_x$. 

 For $x\in \supp(F)$, we get $G_{x}\subset
 R_0$ and thus $G_{x}=R_0$. The end of the lemma is  obvious. 
\end{proof}





\subsection{Multi-type offspring distribution}

We define  a multi-type  offspring distribution  $p$ of  $d$ types  as a
sequence  of  probability  distributions: $p=(p^{(i)},  i\in[d])$,  with
$p^{(i)}=(p^{(i)}(k), k\in \N^d)$ a  probability distribution on $\N^d$.
Denote by $f=(f^{(1)}, \ldots, f^{(d)})$  the generating function of the
offspring distribution 
$p$, i.e.  for $i\in[d]$ and $s\in[0,1]^d$:
\begin{equation}
   \label{eq:genfunc}
f^{(i)}(s)  =\E[s^{X_i}], 
\end{equation}
with  $X_i=(X_i^{(j)},  j\in [d])$  a  random  variable on  $\N^d$  with
distribution    $p    ^{(i)}$.    Denote    by    $m_{ij}=\partial_{s_j}
  f^{(i)}(\ind)=\E[X_i^{(j)}]   \in[0,+\infty]$    the
expected number  of  offsprings with type $j$ of a single individual  of type
$i$.  Denote by  $M$ the mean matrix $M=(m_{ij}; \,i,  j\in[d])$ and set
$(m_{ij}^{(n)};   \,i,  j\in[d])=M^n$   for   $n\in  \N^*$.    Following
\cite[p.184]{AN72}, we say that:
\begin{itemize}
   \item[-] $p$ is non-singular if $f(s)\neq M s$.
\item[-]  $M$ is finite if $m_{ij} <+\infty $ for all $i, j\in[d]$.
\item[-] $M$  is primitive if $M$  is finite  and  there exists $n\in
  \N^*$ such that for all $i, j\in[d]$, $m_{ij}^{(n)}>0$.
 \end{itemize} 
 By the Frobenius theorem, see \cite[p.185]{AN72}, if $M$ is primitive, then
 $M$  has  a  unique  maximal  (for  the  modulus  in  $\C$)  eigenvalue
 $\rho$. Furthermore  $\rho$ is  simple, positive ($\rho\in  (0, +\infty
 )$), and  the corresponding right  and left eigenvectors can  be chosen
 to be positive.  If  $\rho=1$ (resp.   $\rho>1,$ $\rho<1$),  we say  that the
 offspring  distribution  and  the  associated multi-type  GW  tree  are
 critical (resp. supercritical, subcritical).

Recall the definition of an aperiodic  distribution
given in Definition \ref{defi:aperiodic}. 
\begin{defi}
\label{defi:aperiodic-p} 
Let  $p=(p^{(i)},
i\in[d])$ be an offspring distribution.
We say that  $p$ is  aperiodic, if  the smallest subgroup of $\Z^d$ that
contains $\bigcup _{i=1}^d \left(\supp(p^{(i)}) -\supp(p^{(i)})\right)$
is $\Z^d$. 
\end{defi}

For an offspring distribution $p$, we shall consider the following assumptions:
\begin{itemize}
  \item[($H_1$)] \textbf{The mean matrix $M$ of $p$ is primitive, and  $p$ is critical
    and non-singular.}
  \item[($H_2$)] \textbf{The offspring distribution $p$ is aperiodic.}
\end{itemize}

\subsection{Multi-type Galton-Watson tree and Kesten's tree}

We define the multi-type GW tree $\tau$ with offspring distribution $p$.
\begin{defi}
Let $p$ be an offspring distribution of $d$ types and $\alpha$ a
probability distribution on $[d]$. A $\mathbb{T}$-valued random
variable $\tau$ is a multi-type GW  tree with offspring distribution $p$ and root
type distribution $\alpha$, if for
all $h \in \N$, $\bt\in \T^{(h)}$, we have:
\[
\P_\alpha(r_h(\tau)=\bt)=\alpha(\cm(\emptyset_\bt)) \prod_{u\in \bt, |u|<h}
\frac{k_u^{(1)}[\bt]!\cdots k_u^{(d)}[\bt]!}{k_u[\bt]!}
\, p^{(\mathcal{M}(u))}(k_u(\bt)).
\]
\end{defi}

We deduce from the definition that for $\bt\in\mathbb{T}_0$, we have
\[
\P_\alpha(\tau=\bt)=\alpha(\cm(\emptyset_\bt)) \prod_{u\in \bt}
\frac{k_u^{(1)}[\bt]!\cdots k_u^{(d)}[\bt]!}{k_u[\bt]!}\, 
p^{(\mathcal{M}(u))}(k_u(\bt)).
\]
The multi-type GW tree enjoys the branching property: an
individual  of type $i$ generates children according to $p^{(i)}$
independently of any born individual, for $i\in[d]$. 

Let $p$  be an  offspring distribution  of $d$  types such  that ($H_1$)
holds.  Denote by  $a^*$ (resp. $a$) the right  (resp. left) positive
normalized  eigenvector  of $M$  such  that  $\langle a,  \ind  \rangle=
\langle  a,   a^*\rangle=1$. Those eigenvectors correspond to the
eigenvalue $\rho=1$.  Notice   that  $a$  is   a  probability
distribution   on  $[d]$.    The  corresponding   size-biased  offspring
distribution  $\hat{p}=(\hat{p}^{(i)},i\in[d])$   is  defined   by:  for
$i\in[d]$ and $k\in\N^d$,
\begin{equation}
   \label{eq:hat-p}
\hat{p}^{(i)}(k)=\frac{\langle k, a^*\rangle}{ a^*_i}\, p^{(i)}(k) .
\end{equation}
For $\alpha$ a probability distribution on $[d]$, we also define the
corresponding size-biased distribution $\hat \alpha=(\hat \alpha(i),
i\in [d])$ by,   for
$i\in[d]$:
\begin{equation}
   \label{eq:hat-a}
\hat\alpha (i)= \alpha(i) \frac{a^*_i }{\langle \alpha,
  a^*\rangle}\cdot 
\end{equation}

\begin{defi}
  Let  $p$ be  an offspring  distribution of  $d$ types whose mean
  matrix is primitive  and let $\alpha$ be  a
  probability  distribution   on  $[d]$.  A  multi-type   Kesten's  tree
  $\tau^*$  associated   with  the  offspring  distribution   $p$ and with
  the root type
  distribution $\alpha$ is defined as follows:
\begin{itemize}
  \item[-] Marked individuals are normal or special.
  \item[-] The root of $\tau^*$ is special and its type has
    distribution $\hat \alpha$.
  \item[-] A normal individual of type $i\in [d]$ produces only normal
    individuals according to $p^{(i)}$. 
  \item[-] A  special individual of  type $i\in [d]$  produces children
    according  to $\hat{p}^{(i)}$.  One of  those children,  chosen with
    probability  proportional to  $a^*_j$ where  $j$ is  its type,  is
    special.  The others (if any) are normal.
\end{itemize}
\end{defi}

Notice that  the multi-type  Kesten's  tree  is a  multi-type  GW tree (with
$2d$ types).   The
individuals which are special in $\tau^*$ form an infinite spine, say
$\mathbf{v}^*$,    of   $\tau^*$;    and   the    individuals   of
$\tau^*\backslash \mathbf{v}^* $ are normal.

Let $r\in [d]$.  We shall write $\P_r(d\tau)$, resp. $\P_r(d\tau^*)$,
for the distribution of $\tau$, resp.  $\tau^*$, when the type of its
root is $r$ (that is $\alpha=\delta_r$  the Dirac mass at $r$).  From
\cite{KLPP97}, we get that  for $h\in \N$, $\bt\in\mathbb{T}^{(h)}$ with
$\mathcal{M}(\emptyset_\bt)=r$,   and   $x\in   \mathcal{L}_0(\bt)$   with
$|x|=h$ and $\cm(x)=i$:
\begin{equation}
   \label{eq:K-rh}
\P_{r}(r_h(\tau^*)=\bt, \, v_h^*=x)=\frac{a^*_i}{a^*_{r}}\,
\P_{r}(r_h(\tau)=\bt). 
\end{equation}

Notice that  if $M$ is  primitive and  $p$ is critical  or sub-critical,
then a.s.  Kesten's tree $\tau^*$ belongs to $\mathbb{T}_1$. The next
lemma asserts that there are infinitely many individuals of all types on
the infinite spine.

\begin{lem}
\label{lem:t'1}
  Let $p$ be  an offspring distribution of $d$  types satisfying ($H_1$)
  and  $\alpha$ a  probability  distribution on  $[d]$.  Then a.s.   the
  multi-type Kesten tree $\tau^*$ belongs to $\mathbb{T}'_1$.
\end{lem}

\begin{proof}
Recall  that $a^*=(a^*_i, i\in[d])$ is the normalized right
eigenvalue of $M$ such that $\langle a^*, a \rangle=1$.
By  construction, the sequence  $(\mathcal{M}(v_n^*), n\in\N)$ is
a Markov chain on $[d]$ and transition matrix $Q=(Q_{i,j}, \, i,j\in
[d])$ given by 
\[
Q_{i,j}= \P(\mathcal{M}(v_1^*)=j|\, \mathcal{M}(v_0^*)=i)=
\sum_{k=(k_1, \ldots, k_d)\in \N^d}  \frac{ k_j
  a^*_j}{\langle k, a^*\rangle}\, \hat{p} ^{(i)}(k) = \frac{
  a^*_j}{a^*_i}
\, m_{i,j}, 
\]
where we used  \reff{eq:hat-p} for the definition of $\hat p$ and the
definition of the mean matrix $M$ for the last equality. 
Since $a^*$ is positive and $M$ is primitive, we deduce that $Q$ is
also primitive. This implies that the Markov chain $(\mathcal{M}(v_n^*),
n\in\N)$ is recurrent on $[d]$ and hence it visits a.s. infinitely many
times all the states of $[d]$. 
\end{proof}

The   next  lemma   will  be   used  in   the  proof   of     Theorem
\ref{dTheorem}. In  the next lemma,  we shall consider  a leaf $x$  of a
finite tree $\bt$  with type $i$ and  the root of type  $r$. However, we
will only use the case $i=r$ in the proof of  Theorem \ref{dTheorem}.

\begin{lem}\label{key1}
  Let $p$ be  an offspring distribution of $d$  types satisfying ($H_1$)
  and $r\in [d]$.   Let $\tau$ be a GW tree  with offspring distribution
  $p$  and  $\tau^*$ be  a  Kesten's  tree  associated with  $p$.  For
  all   $\bt\in\mathbb{T}_0$   with  $\mathcal{M}(\emptyset_\bt)=r$,  
  $x\in\mathcal {L}_0(\bt)$ with  $\mathcal{M}(x)=i\in[d]$,  and
  $k\in\mathbb{N}^d$ such that $k\geq |\bt|$,         we have:
\begin{equation}
   \label{eq:cond-global}
  \P_{r}(\tau\in\mathbb{T}(\bt,x)\big|\, |\tau|=k)
  =       \frac{a^*_{r}}{ a^*_i}\, \frac{
    \P_i(|\tau|=k-|\bt|+\mathbf{e}_i)}{
    \P_{r}(|\tau|=k)}\, \P_{r}(\tau^*\in\mathbb{T}(\bt,x)) .
\end{equation}
\end{lem}

\begin{proof}
  Since $\tau^*$  has a  unique infinite spine  $\mathbf{v}^*$ and
  $\bt\in\mathbb{T}_0$, we deduce that $\tau^* \in \T(\bt,x)$ implies
  that $x$ belongs to $\mathbf{v}^*$
  and we get in the same spirit of \reff{eq:K-rh} that:
\begin{equation}
   \label{eq:K-T}
\P_{r}(\tau^*\in\mathbb{T}(\bt,x))=\frac{a^*_i}{a^*_{r }}\,
\P_{r}(\tau\in\mathbb{T}(\bt,x)). 
\end{equation}
We have, following the ideas of \cite{AD14a}:
\begin{align*}
\P_{r}(\tau\in\mathbb{T}(\bt,x),\, |\tau|=k)
&=\sum_{\bt'\in\mathbb{T}_0}\P_{r}(\tau=\bt\otimes(\bt',x))
\ind_{\{| \bt\otimes(\bt',x)|=k\}}  \\
&=\sum_{\bt'\in\mathbb{T}_0}\P_{r}(\tau\in\mathbb{T}(\bt,x)) \, 
\P_i(\tau=\bt')\ind_{\{|\bt\otimes(\bt',x)|=k\}}\\
&=\P_{r}(\tau\in\mathbb{T}(\bt,x))\sum_{\bt'\in\mathbb{T}_0}
\P_i(\tau=\bt')\ind _{\{|\bt'|=k-|\bt|+\mathbf{e}_i\}}\\
&=\P_r(\tau\in\mathbb{T}(\bt,x))\P_i(|\tau|=k-|\bt|+\mathbf{e}_i), 
\end{align*}
where we used the branching property of the multi-type GW tree for the
second equality. 
Use \reff{eq:K-T} to deduce \reff{eq:cond-global}. 
\end{proof}

\section{Main results}
\label{sec:main}

\subsection{Conditioning on the total population size}

Recall that under   ($H_1$), we denote by
$a=(a_\ell,\ell\in [d])$ and $a^*=(a^*_\ell, \ell\in [d])$ the
 positive normalized left and right eigenvectors of the mean matrix $M$ associated
with the eigenvalue  $\rho=1$ such that $\langle a, a^*\rangle=\sum
a_i=1$. 
The proof of the following main theorem is given in Section
\ref{sec:proof-main}.

\begin{theo}
\label{dTheorem}
Assume that  ($H_1$) and ($H_2$)  hold.  Let $(k(n), n\in\N^*)$  be a
sequence of $\N^d$ satisfying $\lim_{n\rightarrow\infty}{|k(n)|}=+\infty
$ and  $\lim_{n\rightarrow\infty}{k(n)}/{|k(n)|} =a$.   Let $\tau$  be a
random GW tree with critical  offspring distribution $p$ and root
type distribution $\alpha$, and $\tau_n$ be
distributed  as $\tau$  conditionally  on  $\{|\tau|=k(n)\}$.  Then  the
sequence  $(\tau_n, n\in  \N  ^  *)$ converges  in  distribution to  the
Kesten's tree $\tau^*$ associated with $p$ and $\alpha$.
\end{theo}

\begin{rem}
   \label{rem:miermont}
   Let  $\tau$ be  a critical  GW tree  with offspring  distribution $p$
   satisfying  ($H_1$). We  can consider  $\tau$ conditionally  on the
   event that the population of type $i$, $|\tau^{(i)}|$, is large. According
   to Proposition 4 in \cite{M08}, the random variable $|\tau^{(i)}|$ is
   distributed as the  total number of vertices of  a critical mono-type
   GW  tree under  $\mathcal{M}_{\tau}(\emptyset)=i$,  or  as the  total
   number  of  vertices of  a  random  number of  independent  mono-type
   critical    GW   trees    with    the    same   distribution    under
   $\mathcal{M}_{\tau}(\emptyset)\neq i$. In  particular, we deduce from
   \cite{AD14a}  that,  if  $p^{(i)}$  is aperiodic,  the   key
   equality   $\lim_{n\rightarrow+\infty  }   {  \P(|\tau^{(i)}|=n-b)}/{
     \P_{r}(|\tau^{(i)}|=n)}=1$ holds for any  $b\in \Z$.  And following
   the proof  of Theorem \ref{dTheorem} after  Equation \reff{eq:keyeq},
   we easily get  that $\tau$ conditioned on  $|\tau^{(i)}|$ being large
   converges locally  to Kesten's  tree. See  \cite{S14} for  a detailed
   proof.
\end{rem}

\begin{rem}
   \label{rem:reduction}
   The  local  convergence  of  a multi-type  critical  GW  tree  $\tau$
   conditioned on the  number of vertices of one fixed  type being large
   to a Kesten's tree has been proved in \cite{S14}. It would be easy to
   extend  Theorem  \ref{dTheorem},  with the  same  minimal  conditions
   ($H_1$) and ($H_2$) to a conditioning on an asymptotic proportion per
   types for  $d'$ types,  with $d'<d$ by  using the  constructions from
   \cite{R11} or  from \cite{M08}. The  idea is  to map a  multi-type GW
   tree $\tau$ with  $d$ types onto another GW tree  $\tau'$ with $d'<d$
   types  and  offspring distribution  $p'$  so  that  the size  of  the
   population  of types  $1$  to  $d'$ of  $\tau$  and  $\tau'$ are  the
   same. Then  the key Equation  \reff{eq:keyeq} is now replaced  by the
   one  for  $\tau'$ which  holds  if  the offspring  distribution  $p'$
   satisfies  ($H_1$) and  ($H_2$).  Then  the proof  follows as  in the
   proof of Theorem \ref{dTheorem} after Equation \reff{eq:keyeq}.
\end{rem}

\begin{rem}
   \label{rem:penisson}
   The  change  of  offspring  distribution  given  in  Section  1.4  of
   \cite{P14}, when it exists, allows to extend Theorem \ref{dTheorem} to
   sub-critical multi-type GW trees.  In order to consider an asymptotic
   proportion of types different from the one given by the (positive normalized)
   left eigenvector  associated with the Perron-Frobenius  eigenvalue, one
   has  to  change   the  offspring  distribution,  see   Theorem  3  of
   \cite{P14}.  However,  this  requires  exponential  moments  for  the
   offspring distribution.
\end{rem}

We end this Section by using Theorem \ref{dTheorem} to extend results of
\cite{AD14b} on mono-type GW tree in the following sense.  Let $\tau$ be
a mono-type  GW tree (that  is $d=1$) with critical  aperiodic offspring
distribution  $  q=(  q(\ell),  \ell\in  \N)$.   Let  $f_q$  denote  the
generating function of $q$  and $\cq=\{\gamma>0; f_q(\gamma)<+\infty \}$
its domain on $(0, +\infty )$.

Let $d\geq  2$ and assume that  $\Card(\supp q)\geq d+1$.  Since  $q$ is
critical we have $0\in \supp q$.   Let $A_1, \ldots, A_d$ be a partition
of  $\supp  q$  such  that   $0\in  A_1$  and  $\Card(A_1)>1$.   We  set
$\alpha(i)=\sum_{\ell\in A_i}  q(\ell)$ for all $i\in  [d]$.  Notice that
   $\alpha$   is    a   positive    probability   distribution    and
$\alpha(1)>q(0)$.  We   set  $|\tau|=(|\tau^{(i)}|,  i\in   [d])$  where
$|\tau^{(i)}|$ be  the number  of individuals of  $\tau$ whose  number of
offsprings belongs to $A_i$.

For $x=(x_i, i\in [d])$ we set  $  \mathfrak{m}_{x}:=\sum_{i\in  [d]}
x_i \,
   \inf A_i$ and for $\gamma\in \cq$: 
\[
h_{x}(\gamma)=\sum_{i\in [d]} x_i\, 
\frac{f'_{A_i}(\gamma)}{f_{A_i}(\gamma)} 
\quad\text{with}\quad
f_{A_i}(\gamma)=\sum_{\ell\in   A_i}  \gamma^\ell   q(\ell) \quad
\text{for all $i\in                  [d]$.}
\]

\begin{cor}
   \label{cor:monoGW}
Let $q$ be a  critical aperiodic offspring
distribution. 
   Let $\tilde  \alpha=(\tilde \alpha(i),  i\in [d])\in \R^d$ be  such that
   $\tilde \alpha>0$ and  $\langle \tilde \alpha, \ind  \rangle=1$, so that
   $\tilde   \alpha$  is   a   non-degenerate   proportion.  Assume that:
\begin{equation}
   \label{eq:m<1}
\mathfrak{m}_{\tilde \alpha}<1
\end{equation}
and 
\begin{equation}
   \label{eq:root}
\text{there exists a (unique) $\gamma\in \cq$ such that}
\quad h_{\tilde \alpha}(\gamma)=1. 
\end{equation}
Let   $(k(n),   n\in\N^*)$   be   a  sequence   of   $\N^d$   satisfying
$\lim_{n\rightarrow\infty}{|k(n)|}=+\infty             $             and
$\lim_{n\rightarrow\infty}{k(n)}/{|k(n)|}  =\tilde \alpha$.   Let $\tau$
be  a random  mono-type GW  tree  with offspring  distribution $q$,  and
$\tau_n$ be  distributed as  $\tau$ conditionally  on $\{|\tau|=k(n)\}$.
Then the sequence  $(\tau_n, n\in \N ^ *)$ converges  in distribution to
the  Kesten's   tree  $\tilde   \tau^*$  associated  to   the  offspring
distribution  $\tilde   q=(\tilde  q(\ell),   \ell\in  \N)$   where  for
$i\in [d]$, $\ell\in A_i$:
\[
\tilde q(\ell)=\frac{\tilde \alpha(i)}{f_{A_i}(\gamma)} \, \gamma^\ell 
q(\ell).
\]
\end{cor}
Notice  that if  $\tilde  \alpha=\alpha$,  then condition  \reff{eq:m<1}
holds  as $\Card(  A_1)>1$ and $q$ is critical;  and condition  \reff{eq:root}  also holds  with
$\gamma=1$ as $q$ is critical. We deduce that if $\tilde \alpha=\alpha$,
then $\tilde q=q$ and $\tilde \tau^*=\tau^*$ is simply the Kesten's tree
associated to  the offspring  distribution $q$.  We  now comment  on the
conditions \reff{eq:m<1} and \reff{eq:root}.

\begin{rem}
   \label{rem:m<1}
  One  can see
   that condition \reff{eq:m<1} is almost optimal. This is easy to check
   in  the binary  case.   Assume $q(0)+q(1)+q(2)=1$,  $q(0)q(1)q(2)>0$,
   $A_1=\{0,1\}$    and    $A_2=\{2\}$.     Since   we    always    have
   $|\tau^{(1)}|>|\tau^{(2)}| $,  then any asymptotic proportion  has to
   satisfies   $\tilde  \alpha(1)   \geq  \tilde   \alpha(2)$  that   is
   $\mathfrak{m}_{\tilde \alpha} = 2\tilde \alpha(2)\leq 1$.
\end{rem}

\begin{rem}
   \label{rem:root}
   For $i\in  [d]$, let $\rP_{i,  \gamma}$ denote the distribution  of a
   random   variable   $Z$   taking    values   in   $A_i$   such   that
   $\rP_{i,  \gamma}(Z=\ell)=\ind_{\{\ell\in  A_i\}}\gamma^\ell  q(\ell)
   /f_{A_i}(\gamma)$.
   In                 particular,                we                 have
   $h_{\tilde    \alpha}(\gamma)=\sum_{i\in   [d]}    \tilde   \alpha(i)
   \rE_{i,\gamma}[Z]$.
   An         elementary          computation         gives         that
   $\partial_\gamma  \rE_{i,\gamma}[Z]=\gamma^{-1}  \Var_{i,\gamma}(Z)$.
   Using  that $\Card(A_1)>1$,  we get  $\Var_{1,\gamma}(Z)>0$ and  thus
   $h'_{\tilde \alpha}$  is positive  on $\cq$.  We  deduce that,  if it
   exists,  the root  of  the equation  \reff{eq:root}  is then  unique.
   Since
   $\lim_{\gamma\rightarrow                 0}                 h_{\tilde
     \alpha}(\gamma)=\mathfrak{m}_{\tilde                      \alpha}$,
   we      deduce     that      condition     \reff{eq:m<1}      implies
   $h_{\tilde \alpha}(0+)<1$.  A necessary  and sufficient condition for
   the    existence   of    a   root    to   \reff{eq:root}    is   that
   $\lim_{\gamma\uparrow  R}  h_{\tilde  \alpha} (\gamma)\geq  1$,  with
   $R=\sup  \cq$ the  radius  of  convergence of  the  series $f_q$.   A
   sufficient  condition to  get this  latter condition  is for  example
   $\lim_{\gamma\uparrow R} f'_q(\gamma)=+\infty $  or even the stronger
   condition $R=+\infty $.

   As noticed earlier,  for $\tilde \alpha=\alpha$, as  $q$ is critical,
   we get that  $h_\alpha(1)=1$.  So in this case  no further hypothesis
   are needed.  We also deduce that, if $\tilde \alpha(i)\geq \alpha(i)$
   for  all  $i\in [d]$  such  that  $f_{A_i}(1)\geq  1$, then  we  have
   $h_{\tilde   \alpha}(1)\geq  h_\alpha(1)=1$.    And  thus,   in  this
   particular case also,  the root  of \reff{eq:root} exists  without further
   assumptions on $q$.
\end{rem}

\begin{proof}[Proof of Corollary \ref{cor:monoGW}]
  We consider artificially  that $\tau$ is a  $d$ dimensional multi-type
  GW tree, by  saying that an individual  $u\in \tau$ is of  type $i$ if
  the number of offsprings of $u$  belongs to $A_i$.  The corresponding root
  type  distribution   is  $\alpha$  and  the   corresponding  offspring
  distribution  $p=(p^{(i)},  i\in  [d])$  is defined  as  follows:  for
  $k=(k_1, \ldots, k_d)\in \N^d$,
\begin{equation}
   \label{eq:p_q}
p^{(i)}(k)=\ind_{\{|k|\in A_i\}}
\frac{q(|k|)}{\alpha(i)}\,\,\binom{|k|}{k} \, \alpha^k, 
\end{equation}
where  we recall  that $\alpha^k=\prod_{j\in  [d]} \alpha(j)^{k_j}$  and
$|k|=\sum_{i\in [d]}  k_i$, and  we use the  following notation  for the
multinomial coefficient $\binom{|k|}{k}=|k|!/\prod_{i\in[d]} k_i!$.  For
simplicity we shall still denote the corresponding multi-type GW tree by
$\tau$.    We    define   $\alpha^*=(\alpha^*(i),   i\in    [d])$   with
$\alpha^*(i)=\sum_{\ell\in   A_i}   \ell  q(\ell)/\alpha(i)$   so   that
$\langle  \alpha,  \ind  \rangle=  \langle  \alpha,  \alpha^*\rangle=1$.
Notice that $\alpha^*$  is positive as $\alpha(1)>q(0)$  and $0\in A_1$.
It   is   easy  to   check   that   the   mean   matrix  is   given   by
$M=(\alpha^*)^T \alpha$. Its only non  zero eigenvalue is 1 and $\alpha$
and  $\alpha^*$   are  the   non-negative  associated  left   and  right
eigenvectors.  The mean  matrix $M$ is primitive as all  its entries are
positive.  We get that condition ($H_1$) holds.  Notice ($H_2$) holds as
we assumed $q$ is aperiodic.

We first  consider the  case $\tilde  \alpha=\alpha$.  (As  noticed just
after  Corollary  \ref{cor:monoGW},  condition \reff{eq:m<1}  holds  and
condition \reff{eq:root} also holds  with $\gamma=1$.)  We easily deduce
from Theorem \ref{dTheorem} that if  $(k(n), n\in\N^*)$ is a sequence of
$\N^d$  satisfying   $\lim_{n\rightarrow\infty}{k(n)}/{|k(n)|}  =\alpha$
with $\lim_{n\rightarrow+\infty } |k(n)|=+\infty $, then $\tau_n$, which
is distributed  as $\tau$ conditionally on  $\{|\tau|=k(n)\}$, converges
in distribution to  the ($d$-type) Kesten's tree  $\tau^*$ associated to
the offspring distribution  $p$ and type root  distribution $\alpha$. \\

Let $\hat \tau^*$  be the mono-type Kesten's tree associated  to $q$. We
shall check  that $\tau^*$  is distributed  as $\hat  \tau^*$ seen  as a
multi-type GW  tree, where an individual  $ u\in \hat\tau^*$ is  of type
$i$ if the number of offsprings of $ u$ belongs to $A_i$ and that $u$ is
normal   if  it   has  a   finite   number  of   descendants  (that   is
$\Card  (\{v\in  \hat \tau^*;  \,  u\prec  v\})<+\infty $)  and  special
otherwise.  Let $\bt\in\mathbb{T}_0$ and $x\in\mathcal {L}_0(\bt)$.  For
$i\in     [d]$,      we     set     $x_i=\{\hat     x,      i\}$     and
$\bt_i=\{x_i\}\bigcup (\bt\backslash \{x\})$, the tree which is equal to
$\bt$ except  for the  leaf $\hat  x$ which  is of  type $i$  instead of
$\cm(x)$.   We denote  by  $\T(\bt,\hat x)$  the set  of  trees obtained  by
grafting trees on the leaf $x$  of $\bt$ with possibly changing the type
of $x$,  that is  $\T(\bt,\hat x)=\bigcup _{i\in  [d]} \T(\bt_i,  x_i)$.  We
write
$\P_   {\hat  \alpha}   (d\tau^*)=\sum_{r   \in   [d]}  \hat   \alpha(r)
\P_r(d\tau^*)$. We have for $i\in [d]$:
\[
\P_{\hat \alpha} (\tau^*\in \T(\bt_i,x_i))
= \alpha^*(i) \sum_{r\in [d]} \frac{\hat \alpha(r) }{\alpha^*(r)} \,\P_{r}
(\tau\in \T(\bt_i,x_i)) 
= \alpha^*(i)\,  \P(\tau\in \T(\bt_i,x_i)),
\]
where we used \reff{eq:K-T} for the first equality and \reff{eq:hat-a}
as well as $\langle \alpha, \alpha^* \rangle=1$
for the second.  
Using that 
$\P (\tau\in \T(\bt_i,x_i))=\alpha(i) \P (\tau\in
\T(\bt,\hat x))$ and $\langle \alpha, \alpha^* \rangle=1$, we deduce that:
\[
\P_{\hat \alpha} (\tau^*\in \T(\bt,\hat x))
=\sum_{i\in [d]} \alpha^*(i)\,  \P (\tau\in \T(\bt_i,x_i))
= \P (\tau\in \T(\bt,\hat x)). 
\]
Using  \reff{eq:K-T} in the mono-type case, we get $\P(\hat \tau^*\in
\T(\bt, \hat x))= \P (\tau\in 
\T(\bt,\hat x))$ and thus $\P_{\hat \alpha} (\tau^*\in \T(\bt,\hat x))= 
\P(\hat \tau^*\in
\T(\bt, \hat x))$. It is left to the reader to check that
$\cf'=\{\T(\bt,\hat x);\, \bt\in \T_0, x\in\mathcal 
{L}_0(\bt) \}$ is a separating   class on $\T'_1$. Hence $\tau^*$ is
distributed as $\hat \tau^*$, and can thus
be seen as the (mono-type) Kesten  tree associated with the offspring distribution
$q$.\\

We   now   shall   condition   on  a   general   asymptotic   proportion
$\tilde  \alpha   \in  \R^d$  satisfying  condition   \reff{eq:m<1}  and
condition \reff{eq:root}.   We assume  that there exists  a root  to the
equation \reff{eq:root}, say $\gamma$.  This root is unique according to
Remark \ref{rem:root}.  The probability  $\tilde q$ defined in Corollary
\ref{cor:monoGW} is a critical (as $\gamma$ is a root of \reff{eq:root})
and     aperiodic     (as     $q$     is     aperiodic)     and     that
$\tilde  \alpha(i)=\sum_{\ell\in  A_i}  \tilde q(\ell)$  for  all  $i\in
[d]$.
Let $\tilde  \tau$ be  a mono-type GW  tree with  offspring distribution
$\tilde q$  (which can also  be seen as a  multi-type GW tree  where the
type of an  individual is $A_i$ if  the number of its  offspring lies in
$A_i$). We deduce that for all $\bt\in \T_0$, we have:
\[
\P_{\tilde \alpha}(\tilde \tau=\bt)=\prod_{u\in \bt} \tilde q(k_u[\bt])
= \gamma^{\langle |\bt|, \ind \rangle-1} \Gamma^{|\bt|} \,  \P_\alpha(\tau=\bt),
\]
where    $\Gamma=(\tilde\alpha(i)/f_{A_i}(\gamma),   i\in    [d])$.   In
particular,  for all  $k\in \N^d$,  the random  tree $\tau_n$,  which is
distributed  as $\tau$  conditionally  on $\{|\tau|=k\}$,  has the  same
distribution as the random tree $\tilde \tau_n$, which is distributed as
$\tilde \tau$ conditionally on  $\{|\tilde \tau|=k\}$.  According to the
first part,  since $\tilde q$ is  critical aperiodic, we deduce  that if
$(k(n),    n\in\N^*)$   is    a    sequence    of   $\N^d$    satisfying
$\lim_{n\rightarrow\infty}{k(n)}/{|k(n)|}    =\tilde     \alpha$    with
$\lim_{n\rightarrow+\infty } |k(n)|=+\infty $, then $\tilde \tau_n$, and
thus $\tau_n$, converges in distribution  to the mono-type Kesten's tree
$\tilde \tau^*$ associated to $\tilde q$.
\end{proof}


\subsection{Around the Dwass formula}

   Let  $\tau$ be  a
random GW  tree with  critical offspring distribution  $p$. We have no
assumption on $p$ for the moment. For
$i,j\in[d]$, we define  the total number of individuals of type $i$
whose parent is of type $j$:
\[
B_{ij}=\Card \left(\{u\in \tau, \, \cm(u)=i  \text{ and }  \cm(\fa(u))=j\}\right).
\]
And we set $\cb=(B_{ij}; i,j\in [d])$. Notice that $\sum_{j\in  [d]}
B_{ij}=|\tau^{(i)}|$. 

Let  $(X_{i,  \ell};  \ell\in\mathbb{N}^*)$   for  $i\in[d]$  be  $d$
independent families  of independent random variables  in $\mathbb{N}^d$
with $X_{i, \ell}$ having probability distribution $p^{(i)}$.  For $i\in
[d]$, we consider the random walk $S_{i, n}=\sum\limits_{\ell=1}^n X_{i,
  \ell}$ for $n\in \N^*$  with $S_{i, 0}=0$.  For $k=(k_1,\ldots,k_d)\in
\N^d$,  we  set  $S_{k}=\sum_{i\in  [d]}  S_{i,  k_i}$.   We  adopt  the
following convention for a $d$-dimensional  random variable $X$ to write
$X=(X^{(j)},  j\in  [d])$,   so  that  we  have   in  particular  $S_{i,
  n}^{(j)}=\sum\limits_{\ell=1}^n X_{i,  \ell}^{(j)}$.  For  $k\in \N^d$
and $r\in [d]$, we define the matrix $\cs(k,r)=(\cs_{ij}(k,r); \, i,j\in
[d])$ of size $d\times d$ by:
\begin{equation}
   \label{eq:def-S}
\cs_{ij}(k,r)=-S_{i,k_i}^{(j)}+ (S_k^{(j)}+\ind_{\{r=i\}} )\ind_{\{i=j\}}.
\end{equation}

The following lemma is a direct consequence of the representation of
Chaumont and  Liu \cite{CL13}  for multi-type GW  process, which  
generalizes the  Dwass formula to the multi-type case.

\begin{lem}\label{lem:CandL}
 Let  $\tau$ be  a
random GW  tree with  critical offspring distribution  $p$.   For $r\in
[d]$ and $k\in (\N^*)^d$, we have:
\[
\P_{r}(|\tau|=k)
=\inv{\prod_{i\in [d]} k_i} \E\left[\det(\cs(k,r)); \, S_k+{\bf
      e}_r=k\right].
\]
\end{lem}
\begin{proof}
For $\kappa=(\kappa_{ij}; i,j\in [d]) \in \N^{d\times d}$, we denote,  for $j
\in [d]$, by
$\kappa_j$  the column vector $(\kappa_{ij}, i \in [d])$. 
We deduce from Theorem 1.2 in 
\cite{CL13} that, for $r\in [d]$, $k=(k_1, \ldots, k_d) \in (\N^*)^d$,
$\kappa=(\kappa_{ij}; i,j\in [d]) \in \N^{d\times d}$ such that 
\begin{equation}
   \label{eq:cond-kappa}
k={\bf e}_r+\sum_{j\in [d]}\kappa_{ j}, 
 \end{equation} 
we have:
\begin{equation}
\label{ChaumontandLiu}
\P_r(\cb=\kappa)=\det(\Delta(k)- \kappa)\, 
\prod_{j\in [d]}\frac{\P(S_{j, k_j}=\kappa_j)}{k_j},
\end{equation}
where $\Delta(k)$ is the $d\times  d$ diagonal matrix with diagonal $k$.
Notice that additional hypotheses on the offspring distribution $p$ were
required in Theorem 1.2 from  \cite{CL13}.  However, for fixed $\kappa$,
\reff{ChaumontandLiu} is a finite algebraic expression of $p$. According
to  \cite{CL13}, it  holds in  particular for  all $p$  such that  there
exists a finite constant $c\geq 2$ and $p^{(i)}(k)>0$ if $|k|\leq c$ and
$p^{(i)}(k)=0$  if  $|k|>c$   for  all  $i\in  [d]$.   This  gives  that
\reff{ChaumontandLiu} holds for all $p$.

Because of \reff{eq:cond-kappa}, we have:
\begin{equation}
   \label{eq:|t|=k}
\P_r(|\tau|=k, \, \cb=\kappa)=\P_r(\cb=\kappa).
\end{equation} 
Thanks to the definition of $\cs(k,r)$, we have that $\Delta(k)
-\kappa$ is equal to the transpose of $\cs(k,r)$ on $\bigcap _{j\in [d]}
\left\{S_{j,k_j}=\kappa_j\right\}$. 
By summing  \reff{eq:|t|=k} and thus \reff{ChaumontandLiu}  over all the
possible values  of $\kappa$  such that \reff{eq:cond-kappa}  holds, we
get:
\begin{align*}
\P_r(|\tau|=k)
&=\sum_{\kappa } \P_r(\cb=\kappa) \ind_{\{k={\bf e}_r+\sum_{j\in [d]}\kappa_{ j}\}}\\
&=\inv{\prod_{j\in [d]} k_j} \sum_{\kappa } \det(\Delta(k)- \kappa)\, 
\ind_{\{k={\bf e}_r+\sum_{j\in [d]}\kappa_{ j}\}}
\P(\forall j\in [d],\, S_{j, k_j}=\kappa_j )\\
&=\inv{\prod_{i\in [d]} k_i}
 \E\left[\det(\cs(k,r)); \, {\bf e}_r+S_k=k  \right].
\end{align*}
\end{proof}

In order to compute the determinant $\det(\cs(k,r))$, instead of using a
development based on permutations, we shall use a development based on
elementary forests, see Lemma 4.5 in \cite{CL13} and Formula
\reff{eq:CL13} below. (As we are interested in computing the determinant
of a matrix whose all columns  but one sum up to 0, we shall only consider
forests reduced to one tree.)

Recall $\ind=(1,\ldots,1)\in\R^d$. For $r\in  [d]$, we consider $\Tau_r$
the subset of $\T_0$  of trees with root of type  $r$, and having exactly $d$
individuals, all of them with a distinct type:
\[
\Tau_r=\{\bt\in \T_0;\,  |\bt|=\ind, \text{ and }
\cm(\emptyset_\bt)=r\}. 
\]
For $\bt\in \Tau_r$ and $j\in [d]\setminus\{r\}$, let $j_\bt$ denote the
type of the parent of the individual of type $j$: $j_\bt=\cm(\fa(u_j))$,
where $u_j$ is the only element of $\bt$ such that $\cm(u_j)=j$. We
shall use the
following formula to give asymptotics on $\det(\cs(k,r))$. 

\begin{lem}
   \label{lem:detS}
For $r\in [d]$ and $k\in (\N^*)^d$, we have:
\[
\det(\cs(k,r))= \sum_{\bt\in  \Tau_r}\, \,  \prod_{j\in [d]\setminus
  \{r\} }
S^{(j)}_{j_\bt,k_{j_\bt}}.
\]
\end{lem}

\begin{proof}
  We  follow the  presentation of  \cite{CL13}. We  say that a  collection of
  trees is   a forest. A  forest $\bff=(\bt_j, j\in J)$  is called
  elementary  if the  trees  are  pairwise disjoint  and  if the  forest
  contains exactly one individual of each type, that is $\sum_{j\in J}
  |\bt_j|=\ind$. 
Let $\F$  denote the set  of elementary  forests. For $\bff\in  \F$, set
$u_i$ the individual in  $\bff$ of type $i$, which belongs  to a tree of
$\bff$ say $\bt_j$, and write $i_\bff=\cm(v)$ for the type of the parent
$v=\fa_{u_i}(\bt_j)$ of $u_i$ if $|u_i|>0$ and $i_\bff=0$ if $|u_i|=0$.

According to Lemma 4.5 in \cite{CL13}, we have  for
$\kappa=(\kappa_{ij}; i,j\in [d]) \in \R^{d\times d}$
\begin{equation}
   \label{eq:CL13}
\det(\kappa)=(-1)^d\sum_{\bff\in  \F} \prod_{j\in [d]}
\kappa_{j_\bff,j},
\end{equation}
with the convention that $\kappa_{0,j}=- \sum_{i\in [d]} \kappa_{ij}$.

Thanks  to  Definition \reff{eq:def-S}  of $\cs(k,r)$,  this implies
that for $r\in [d]$ and $k\in (\N^*)^d$, we have:
\begin{equation}
   \label{eq:detS}
\det(\cs(k,r))= \sum_{\bff\in  \F} \prod_{j\in [d]}
S^{(j)}_{j_\bff,k_{j_\bff}},
\end{equation}
with      the     convention      that      if     $j_\bff=0$,      then
$S^{(j)}_{j_\bff,k_{j_\bff}}= \ind_{\{j=r\}}$. Notice that
$\prod_{j\in [d]}
S^{(j)}_{j_\bff,k_{j_\bff}}=0$ if the forest $\bff$ is not reduced to
a single tree whose root is of type $r$. 
To  conclude,  use that $j_\bff=j_\bt$ if the forest $\bff$ is
reduced to a single tree $\bt$.
\end{proof}

Let $(\tilde  X_{i, \ell};  \, \ell\in\mathbb{N}^*,  i\in [d])$  be a
sequence  of  random  variables  independent   of  $(  X_{i,  \ell};  \,
\ell\in\mathbb{N}^*, i\in [d])$ with the same distribution.

For  a  finite  subset  $K$   of  $\N$,  we  shall  consider  partitions
$\textbf{A}^{(\ell,K)} =(A_1^{K},  \ldots , A_\ell^K)$ of  $K$ such that
$\inf A_1^K< \cdots  <\inf A_\ell^K$.  For $\bt\in  \Tau_r$, $i\in [d]$,
recall that  $u_i$ is  the individual  in $\bt$ of  type $i$.  Denote by
$C_i(\bt)=\{j\in [d]; \, j_\bt=i\}$ the set  of types of the children of
$u_i$    in   $\bt$.    Let   $\A_\bt$    be   the    family   of    all
$\ca=(m,   (    \textbf{A}^{(m_i,   C_i(\bt))}),   i\in    [d])$,   with
$m=(m_1, \ldots, m_d)\in  \N ^d$ such that, for all  $i\in [d]$, $m_i=0$
if   $\Card(C_i(\bt))=0$   and   $1\leq  m_i\leq   \Card(C_i(\bt))$   if
$\Card(C_i(\bt))>0$.   For   convenience,  we  may  write   $m_\ca$  for
$m$. With this notation, we set:
\[
\tilde S_{m_\ca} =\sum_{i\in [d]} \, \sum_{\ell=1}^{m_i} 
\tilde X_{i,\ell}, 
\quad
G(\ca)=\prod_{i\in [d]}\, \, \prod_{\ell=1}^{m_i} \, \prod_{j\in
  A_\ell^{C_i(\bt)}} \tilde X_{i,\ell}^{(j)},
\]
with the convention that $\sum_\emptyset=0$ and $\prod_\emptyset=1$, and for $k=(k_1, \ldots, k_d)\in \N^d$ such that $k_i\geq d$ for
all $i\in [d]$:
\[
B_k(m_\ca)= \prod_{i\in [d]} \frac{k_i!}{(k_i-m_i)!} \cdot
\]
Since  $\tilde X_{i,\ell}$ for $i\in [d]$, $\ell\in \N^*$
takes values in $\N^d$ and $\sum_{i\in [d]} \sum_{\ell=1}^{m_i} \Card (A_\ell^{C_i(\bt)})=d-1$, we deduce that:
\begin{equation}
   \label{eq:G<H}
0\leq G( \ca)\leq  \val{\tilde S_{m_\ca}}^{d-1}.
\end{equation}

We have the following result.

\begin{cor}
   \label{cor:calc-det}
For $r\in [d]$ and $b, k\in (\N^*)^d$ such that $k\geq
d \ind$, we have:
\[
\E\left[\det(\cs(k,r)); \, S_k=b\right]
= \sum_{\bt\in  \Tau_r}\, \, \sum_{\ca\in \A_\bt} B_k(m_\ca) \, \E\left[ G(\ca); \tilde S_{m_\ca} +
  S_{k-m_\ca}=b \right] .
\]
\end{cor}

\begin{proof}
   For $r\in [d]$, $\bt\in \Tau_r$, and $k\in (\N^*)^d$, we have:
\[
\prod_{j\in [d]\setminus   \{r\} }
S^{(j)}_{j_\bt,k_{j_\bt}}
=\prod_{i\in [d]} \, \prod_{j\in C_i(\bt)} \sum_{\ell=1}^{k_i} X_{i,
  \ell}^{(j)}. 
\]
Using the exchangeability of  $( X_{i,  \ell};  \ell\in\mathbb{N}^*)$
for all $i\in
[d]$, we easily get for $b,k\in (\N^*)^d$ such that $k\geq
d \ind$: 
\[
\E\left[\prod_{j\in [d]\setminus   \{r\} }
S^{(j)}_{j_\bt,k_{j_\bt}}; \, S_k=b\right]
= \sum_{\ca\in \A_\bt} B_k(m_\ca) \, \E\left[ G(\ca); \tilde S_{m_\ca} +
  S_{k-m_\ca}=b \right] .
\]
Then use   Lemma  \ref{lem:detS}  to conclude. 
\end{proof}

\subsection{Proof of Theorem \ref{dTheorem}}
\label{sec:proof-main}
We assume that ($H_1$) and ($H_2$) hold.
Let $r\in  [d]$. Let $b\in \N^d$. We have, using Lemma  \ref{lem:CandL} and  Corollary
\ref{cor:calc-det},  for every  $k\geq b+\ind$ such that $\P_{r}(|\tau|=k)>0$:
\begin{multline}
\label{eq:ratio}
\frac{\prod_{i\in [d]} (k_i-b_i) }{\prod_{i\in [d]} k_i} \, \frac{ \P_r(|\tau|=k-b)}{
    \P_{r}(|\tau|=k)}\\
\begin{aligned}
&=   \frac{
  \E\left[\det(\cs(k-b,r)); \, 
  S_{k-b} +{\bf
      e}_r=k-b\right]}
{  \E\left[\det(\cs(k,r)); \, S_k+{\bf
      e}_r=k\right]}\\
&=  \frac{\sum_{\bt\in  \Tau_r}\, \,
\sum_{\ca\in \A_\bt} B_{k-b}(m_\ca) \, \E\left[ G(\ca); \tilde S_{m_\ca} +
  S_{k-b-m_\ca}=k-b-{\bf
      e}_r \right]
}
{\sum_{\bt\in  \Tau_r}\, \, \sum_{\ca\in \A_\bt} B_k(m_\ca) \, \E\left[
    G(\ca); \tilde S_{m_\ca} + 
  S_{k-m_\ca}=k-{\bf
      e}_r \right]}\cdot
\end{aligned}
\end{multline}

The next lemma  is an extension of the strong  ratio limit theorem given
in    \cite{AD14b}.    Its    proof     is    postponed    to    Section
\ref{sec:proof-dL}.  Recall that  $a$  is the  positive normalized  left
eigenvector of the mean matrix $M$.  (Notice that no moment condition is
assumed for $G$ or $H$.)

\begin{lem}
\label{lem:srlt-HG}
Assume that ($H_1$) and ($H_2$) hold.
 Let $G$ and $H$ be two  random variables in $\mathbb{N}$ and
 $\mathbb{N}^d$ respectively,  independent of $(X_{i,\ell};\,
 \ell\in\N^*, i\in[d])$ and such that $\P(G=0)<1$ and a.s. $G\leq
 |H|^{c}$ for some $c\geq 1$.

Set  $(k(n),n\in\N^*)$ and $(s_n,n\in\N^*)$ be two sequences in
$\mathbb{N}^d$ satisfying $\lim_{n\rightarrow\infty} |k(n)|=+\infty $
and  $\lim_{n\rightarrow\infty}{k(n)}/{|k(n)|}
 =\lim_{n\rightarrow\infty}{s_n}/{|k(n)|}=a$.
Then for any given $m$, $b$ $\in\mathbb{N}^d$, we have:
\[
\lim_{n\rightarrow\infty}
\frac{\E [G;\, H+S_{k(n)-m}=s_n-b]}{
\E [G;\, H+S_{k(n)}=s_n]}=1.
\]
\end{lem}

Let  $(k(n),n\in\N^*)$ be  a sequence  of elements  in $\mathbb{N}^d$
such    that    $\lim_{n\rightarrow\infty}    |k(n)|=+\infty    $    and
$\lim_{n\rightarrow\infty}{k(n)}/|k(n)| =a$.  Since $\P(G(\ca)=0)<1$ and
thanks to \reff{eq:G<H}, we deduce from Lemma \ref{lem:srlt-HG} that:
\[
\lim_{n\rightarrow+\infty }
\frac{\E\left[ G(\ca); \tilde S_{m_\ca} +
  S_{k(n)-b-m_\ca}=k(n)-b-{\bf
      e}_r \right]}
{\E\left[
    G(\ca); \tilde S_{m_\ca} + 
  S_{k(n)-m_\ca}=k(n)-{\bf
      e}_r \right]} =1.
\]
We also have:
\[
\lim_{n\rightarrow+\infty }
\frac{B_{k(n)-b}(m_\ca)}{B_{k(n)}(m_\ca)}=1.
\]
Since all the terms in \reff{eq:ratio} are non-negative, and 
$\lim_{n\rightarrow+\infty } \prod_{i\in [d]} (k_i(n)-b_i)/k_i(n)=1$, 
we deduce that:
\begin{equation}
   \label{eq:keyeq}
\lim_{n\rightarrow+\infty }
 \frac{ \P_r(|\tau|=k(n)-b)}{
    \P_{r}(|\tau|=k(n))}=1.
\end{equation}
Then, using Lemmas \ref{key1}  (with $i=r$ in \reff{eq:cond-global}), we
obtain that, for all $r\in [d]$, $\bt\in \T_0$ and $x\in \cl_0(\bt)$ such
that    $\cm(x)=r$,    $\lim_{n\rightarrow+\infty    }    \P_r(\tau_n\in
\T(\bt,x))=\P_r(\tau^*\in  \T(\bt,x))$.  Of  course  we have  for $\bt\in
\T_0$ and $n$  large
enough that $\P_r(\tau_n=\bt)=0=\P_r(\tau^*=\bt)$.  We
deduce from  Corollary \ref{cor:cv} that $(\tau_n,  n\in \N^*)$ converges
in distribution towards $\tau^*$ under  $\P_r$ for all $r\in [d]$.  

Let $\alpha$ be  a
probability distribution on $[d]$. Let $\bt\in \T_0$ and $x\in \cl_0(\bt)$.
Set $r=\cm(\emptyset_\bt)$ and $i= \cm(x)$. We have using \reff{eq:cond-global} that:
\begin{align*}
   \P_{\alpha}(\tau_n\in\mathbb{T}(\bt,x)) 
&= \frac{\alpha(r)\, \P_r (\tau\in\mathbb{T}(\bt,x), \,
  |\tau|=k(n))}{\sum_{j\in
  [d]} \alpha(j) \,  \P_j( |\tau|=k(n))}
= \frac{\alpha(r) a^*_r}{\sum_{j\in
  [d]} \alpha(j) a^*_i \Gamma_j(n)}\, \P_r (\tau^*\in\mathbb{T}(\bt,x)),
\end{align*}
where 
\[
\Gamma_j(n)=\frac{\P_j(|\tau|=k(n))}{\P_i(|\tau|=k(n) - |\bt| 
  +\mathbf{e}_i)},
\]
with the convention that $\Gamma_j(n)=+\infty $ if $\P_i(|\tau|=k(n) - |\bt| 
  +\mathbf{e}_i)=0$. 
Let $\bt'\in \T_0$ and $x'\in \bt'$ such that $\cm(x)=i$ and
$\P_j(\tau^*\in \T(\bt', x'))>0$ (which is possible thanks to Lemma
\ref{lem:t'1}). Using \reff{eq:cond-global} and the convergence of $(\tau_n,  n\in \N^*)$  towards $\tau^*$ under  $\P_j$, we deduce that
$\lim_{n\rightarrow+\infty } 
\frac{\P_j(|\tau|=k(n))}{\P_i(|\tau|=k(n) - |\bt'| 
  +\mathbf{e}_i)}=a^*_j/a^*_i$. Then use \reff{eq:keyeq} (with $r=i$) to
deduce that $\lim_{n\rightarrow+\infty } \Gamma_j(n)=a^*_j/a^*_i$ for all
$j\in [d]$. Using the definition \reff{eq:hat-a} of $\hat \alpha$, we deduce that:
\[
\lim_{n\rightarrow+\infty }   \P_{\alpha}(\tau_n\in\mathbb{T}(\bt,x)) =
\frac{\alpha(r) a^*_r}{\sum_{j\in
  [d]} \alpha(j) a^*_j}\, \P_r (\tau^*\in\mathbb{T}(\bt,x))
= \P_{\hat \alpha} (\tau^*\in\mathbb{T}(\bt,x)),
\]
where $\P_ {\hat \alpha} (d\tau^*)=\sum_{r \in [d]} \hat \alpha(r)
\P_r(d\tau^*)$.

\subsection{Proof of Lemma \ref{lem:srlt-HG}}
\label{sec:proof-dL}
We assume ($H_1$). In particular, this implies that  $\P(X_{i,1}=0)>0$
for some $i\in [d]$. Without loss of generality, we can assume this
holds for $i=d$: $\P(X_{d,1}=0)>0$. 

Recall that  $a$ is the normalized  left positive eigenvector of the  mean matrix
$M$ such that $|a|=1$. In particular $a$ is a probability on $[d]$.  Set
$\mathbf{v}_d=0\in    \N^{d-1}$    and    for    $i\in    [d-1]$,    set
$\mathbf{v}_i=(v_i^{(1)},  \ldots, v_i^{(d-1)})\in  \N^{d-1}$ such  that
$v_i^{(j)}=\ind_{\{j=i\}} $ for $j\in [d-1]$.  Let $Y=(U,V)$ be a random
variable  in   $  \N^d\times  \N^{d-1}$   such  that  for   $i\in  [d]$,
$\P(V=\mathbf{v}_i)=a_i$, and  the distribution of $U$  conditionally on
$\{V=\mathbf{v}_i\}$ is $p^{(i)}$.

Recall Definition \ref{defi:aperiodic} of an aperiodic probability distribution.
\begin{lem}
   \label{lem:Y-H2}
Under ($H_2$), the distribution of $Y$ on $\Z^{2d-1}$ is aperiodic.
\end{lem}

\begin{proof}
  Recall $\mathbf{v}_d=0\in \N^{d-1}$. Let  $H$ be the smallest subgroup
  of $\Z^{2d-1}$  that contains  $\supp(F) -\supp(F)$,  with $F$  be the
  probability  distribution of  $Y$.  In  particular, we  have that  $H$
  contains  $(\supp(p  ^{(i)})-\supp(p ^{(i)}))\times  \{\mathbf{v}_d\}$
  for  all $i\in  [d]$ and  thus their  union. Since  ($H_2$) holds,  we
  deduce that $H$ contains  $\Z^d \times \{\mathbf{v}_d\}$. This implies
  also that $(0, \mathbf{v}_i)$  belongs to $H$ for all $i\in  [d]$, and thus
  $H=\Z^{2d-1}$.
\end{proof}

For $x\in \R^d$ and $z=(z_1, \ldots, z_d)\in \R^d$, we set
$\delta(x,z)=(x, z_1, \ldots, z_{d-1})$. By definition of
$Y$ and since $a$ is the left eigenvector of the mean matrix with
eigenvalue 1, we have
$\E[Y]=\delta(a,a)$. 

We consider $(Y_\ell, \ell\in \N^*)$ independent random variables
distributed as $Y$. We set $W_n=\sum_{\ell=1}^n Y_\ell$. Let $s\in
\N^d$ and $k\in (\N^*)^d$. We have:
\begin{equation}
   \label{eq:SY=S}
\P\left(W_{|k|}=\delta(s,k)\right) =D(k)\P(S_{k}=s)
\quad\text{with}\quad
D(k)=\frac{|k|!}{\prod_{i\in     [d]}     k_i!}    \prod_{i\in     [d]}
{a_i}^{k_i}. 
\end{equation}

Recall $G$ and $H$ given in Lemma \ref{lem:srlt-HG}. 
We set $H'=\delta(H,0) \in \N^{2d-1}$. We get for $k$,
$m$, $s$ and $b$ in $\N^d$:
\begin{equation}
   \label{eq:S=DW}
\frac{\E [G;\, H+S_{k-m}=s-b]}{
\E [G;\, H+S_{k}=s]}
=\frac{D(k)}{D(k-m)}\, 
\frac{\E [G;\, H'+W_{|k|-|m|}=\delta(s,k)-\delta(b,m)]}{
\E [G;\, H'+W_{|k|}=\delta(s,k)]}\cdot
\end{equation}

Thanks to Lemma  \ref{lem:Y-H2} and ($H_2$), the distribution  of $Y$ on
$\Z^{2d-1}$  is aperiodic.   Since $0\leq  G\leq |H|^c$, we  also have
$0\leq    G\leq   |H'|^c$     and   $\P(G=0)<1$.     Let
$(k(n),n\in\N^*)$  and   $(s_n,n\in\N^*)$  be  two   sequences  in
$\mathbb{N}^d$                                                satisfying
$\lim_{n\rightarrow\infty}|k(n)|=+\infty $ and
$\lim_{n\rightarrow\infty} k(n)/|k(n)|
=\lim_{n\rightarrow\infty}s_n/|k(n)|=a$. Notice, this implies
that $\lim_{n\rightarrow\infty} \delta(s_n, k(n))/|k(n)|=\E[Y_1]$. We
deduce from  Lemma \ref{lem:srlt-H'G} that:
\[
\lim_{n\rightarrow+\infty } \frac{\E [G;\, H'+W_{|k(n)|-|m|}=\delta(s_n,k(n))-\delta(b,m)]}{
\E [G;\, H'+W_{|k(n)|}=\delta(s_n,k(n))]}=1.
\]
Then notice that $\lim_{n\rightarrow+\infty } D(k(n))/D(k(n)-m)=1$ as
$\lim_{n\rightarrow+\infty } k(n)/|k(n)|=a$. And use \reff{eq:S=DW} to get:
\[
\lim_{n\rightarrow+\infty}
\frac{\E [G;\, H+S_{k(n)-m}=s_n-b]}{
\E [G;\, H+S_{k(n)}=s_n]}=1.
\]
This ends the proof of Lemma \ref{lem:srlt-HG}.

\section{Appendix}

\subsection{Preliminary results}
\label{sec:prel}
For $x\in \R^d$  and $\delta\geq 0$, let $\cb(x,\delta)$ be  the open ball
of  $\R^d$ centered  at $x$  with  radius $\delta$.   For any  non-empty
subset $A$ of $\R^d$,  denote: $\cv A$ the convex hull  of $A$, $\clo A$
the closure of  $A$, $\inter A$ the  interior of $A$, $\aff  A=x_0+ \Span(A-x_0)$ the affine
hull of $A$ where  $x_0\in A$ and, if $A$ is convex,  $\ri A$ the  relative interior  of $A$:
\[
\ri A=\{x \in A; \, \aff A \bigcap \cb(x,\delta) \subset A \text{ for some
  $\delta>0$}\}.
\]
Notice that,  for $A$ convex,  we have $\inter A=\ri  A$ if and  only if
$\aff  A=\R^d$.  For  a  function $f$  on $\R^d$  taking  its values  in
$\R\bigcup   \{+\infty  \}$,   its   domain is defined by  $\dom(f)=\{x\in\R^d:
f(x)<\infty\}$.

Let $F$  be a  probability distribution  on $\R^d$ and  $X$ be  a random
variable  on $\R^d$  with  distribution $F$.  Denote  by $\supp(F)$  the
closed support  of $F$:  $x\not \in  \supp(F)$ if  and only  if $\P(X\in
\cb(x,\delta))=0$  for  some $\delta>0$.  Denote  also  by $\cv(F)$  the
convex hull of  its support, $\aff(F)$ and $\ri(F)$ the  affine hull and
the relative interior of $\cv(F)$.   We define $\varphi$ the log-Laplace
of $X$ taking values in $(-\infty , +\infty ]$ as:
\begin{equation}
   \label{eq:def-fi}
\varphi(\theta)=\log\left(\E\left[\expp{\langle\theta, X\rangle}\right]\right),
\quad \theta\in \R^d.
\end{equation}
The  function $\varphi$  is convex,  $\varphi(0)=0$ (which  implies that
$\varphi$  is  proper),  and  lower-semicontinuous  (thanks  to  Fatou's
lemma).  Its conjugate, $\psi$, is defined by:
\begin{equation}
   \label{eq:psi}
\psi(x)=\sup_{\theta\in \dom(\varphi)} \left(  \langle\theta, x\rangle -\varphi(\theta)
\right), \quad x\in \R^d.
\end{equation}
We  recall  that  $\psi$   is  a  lower-semicontinuous  (proper)  convex
function. Since  $\varphi(0)=0$, we deduce that  $\psi$ is non-negative.
We first give a general lemma on the domain of $\psi$.

\begin{lem}
   \label{lem:ri-F}
   Let  $F$   be  a   probability  distribution   on  $\R^d$.   We  have
   $\ri(F)=\ri    \dom(\psi)$.    If    $\psi(x)=0$,   then    we   have
   $x\in \ri \dom(\psi)$.
\end{lem}

\begin{proof}
Let $x\not \in \clo \ri(F)=\clo \cv(F)$.  According to the separation  theorem, there
  exists $\theta\in \R^d$ and $\varepsilon>0$ such that a.s. $\langle
  \theta, X-x \rangle\leq - \varepsilon$. This gives that for all $t>0$,
  $\varphi(t\theta) - t \langle \theta, x \rangle \leq - t\varepsilon$ and thus
$\psi(x) \geq \sup_{t>0} t\varepsilon=+\infty $. This implies that
$\dom(\psi) \subset \clo \ri(F)$. 

Let $x\in \ri(F)$. By translation  invariance, we can assume that $x=0$.
We set  $h(\theta)=\E[\max(0,\min(1, \langle \theta, X  \rangle))]$. The
function     $h$      is     continuous     and, since $0\in \ri(F)$, it
is non      zero     on
$\ca=\{\theta\in  \aff(F),  \, |\theta|=1\}$. Thus  $h$ has a  strictly
positive           minimum           on           $\ca$.           Since
$\P(\langle   \theta,    X   \rangle>0)\geq    h(\theta/|\theta|)$   for
$\theta\neq 0$, we deduce that $a=\inf_{\theta \in \aff(F) \backslash  \{0\}}
\log( \P(\langle   \theta,    X   \rangle>0))$ is finite. 
For $\theta\in \R^d$, let $\theta_F$ denote its  orthogonal projection
on $\aff(F)$. If $\theta_F=0$, then $\varphi(\theta)=0$, otherwise we have
$\varphi(\theta)=\varphi(\theta_F)\geq  \log( \P(\langle   \theta_F,    X
\rangle>0))\geq a$. We deduce that $\varphi\geq a$ and we get
$\psi(x)=\psi(0)\leq
-a$. We deduce that $x\in \dom(\psi)$. This implies that
$\ri(F) \subset \dom(\psi)$. 

We deduce that $\ri(F) \subset \dom(\psi) \subset \clo \ri(F)$, which
gives that $\ri(F) = \ri \dom(\psi) $.

We denote  by $\partial  (F)=\clo \ri( F)\setminus  \ri F$  the relative
boundary of $\dom(\psi)$.  Let $x\in \partial (F)$. Let $X$  be a random
variable with probability distribution  $F$. According to the separation
theorem,    there     exists    $q\in     \R^d$    such     that    a.s.
$\langle  q,  X-x  \rangle\leq  0$   and  $\P(\langle  q,  X-x  \rangle<
0)>0$.
This   implies  that   $\varphi(q)<  \langle   q,x  \rangle$   and  thus
$\psi(x)\geq  \langle  q,x  \rangle  - \varphi(q)>0$.  This  gives  that
$\psi(x)=0$ implies $x\in \ri \dom(\psi)$.
\end{proof}

We have the following corollary.
\begin{cor}
   \label{cor:ri-F}
Let $X$ be a random
variable on $\R ^d $ with probability distribution $F$. If $X$ is integrable
then $\E[X]$ belongs to $\ri \dom (\psi)$ and $\psi(\E[X])=0$.
\end{cor}
\begin{proof}
  Jensen's     inequality    implies     that
  $\varphi(\theta)\geq  \langle  \theta,   \E[X]  \rangle$.  This  gives
  $\langle  \theta, \E[X]  \rangle  -\varphi(\theta) \leq  0$. Then  use
  \reff{eq:psi} and  that  $\psi$ is non-negative to deduce that
  $\psi(\E[X])=0$. Use
  Lemma     \ref{lem:ri-F} to conclude.
\end{proof}

For $\theta\in \dom(\psi)$, we define a probability measure on $\R^d$ by:
\begin{equation}
   \label{eq:P_theta}
d\P_{\theta}(X\in dx)=\expp{\langle \theta,X\rangle-\varphi(\theta)}
d\P(X\in dx).
\end{equation}
We denote by
$m_\theta$ and $\Sigma_\theta$ the corresponding mean vector and
covariance matrix if they exist, i.e:
\begin{equation}
   \label{eq:m-s}
m_\theta=\E_\theta[ X]=\E [X\expp{\langle\theta,
  X\rangle-\varphi(\theta)}]=\nabla \varphi(\theta)
\quad\text{and}\quad \Sigma_\theta=\Cov_\theta(X, X).
\end{equation}
We set $\ci_F=\inter \dom(\varphi)$ the interior of the domain of the
log-Laplace of $F$. Notice that $X$ under $\P_\theta$
has small exponential moment for
$\theta\in \ci_F$ and its mean and covariance matrix are thus
well-defined for $\theta\in \ci_F$. For a symmetric positive
semi-definite matrix $\Sigma$, we denote by $\val{\Sigma}$ its
determinant. The elementary proof of the next lemma is left to the
reader. 

\begin{lem}
\label{lem:L3b}
Let $F$ be a probability distribution on $\R^d$. For any compact set $K \subset
\ci_F$, we have:
\begin{equation}
   \label{eq:L3bound}
\sup_{\theta\in K} |\Sigma_\theta| <+\infty
\quad\text{and}\quad
\sup_{\theta\in K}  \E_\theta\left[|X-m_\theta|^3\right] <+\infty .
\end{equation}
\end{lem}

We set $\co_F=\inter \cv(F)$ the interior of the convex hull of the
support of $F$.
\begin{lem}
   \label{lem:full}
Assume $\co_F$ is non-empty and bounded. Then the application $\theta
\mapsto m_\theta$ is one-to-one from $\R^d$ onto $\co_F$ and continuous
as well as its inverse. In particular, for any compact
set $K\subset\co_F$, there exists $r$ such that $K\subset \{m_\theta; \, |\theta|\leq r\}$.
\end{lem}

\begin{proof}
  It is easy to check, using Hölder's inequality, that if  $\co_F$ is
  non-empty then $\varphi$ is strongly convex on its domain. If $\co_F$
  is bounded, then $X $ is also bounded and the function $\varphi$ is
  finite on $\R  ^d$, so that $\dom(\varphi)=\R^d$, as well as
  differentiable throughout $\R^d$. This implies that
  $\varphi$ is smooth on $\R^d$ in the sense of \cite{c:ca} Section
  26. According to Theorem 26.5 in \cite{c:ca}, this implies that
  $\nabla \varphi$ is one-to-one from $\R^d$ onto the open set $D=\nabla
  \varphi(\R^d)$,  continuous, as is $\nabla
  \varphi^{-1}$. Furthermore, according to Corollary 26.4.1  in
  \cite{c:ca}, we have $\ri \dom(\psi) \subset D \subset
  \dom(\psi)$. Since $D$ is open, we deduce that $D= \ri \dom(\psi)=
  \inter \dom(\psi)$. Then, use Lemma \ref{lem:ri-F} to get that
  $D=\ri(F)=\co_F$.
\end{proof}

Recall Definition \ref{defi:aperiodic} for an aperiodic probability distribution.
\begin{lem}
   \label{lem:dound-i}
Assume $F$ is an aperiodic probability distribution on $\Z^d$. Then, we
have  that $\co_F$  is non-empty and that for any compact set $K \subset \ci_F$,
\begin{equation}
   \label{eq:Lwbound}
\inf_{\theta\in K} |\Sigma_\theta| >0.
\end{equation}
\end{lem}

\begin{proof}
Since $F$ is aperiodic, we have  $\aff(F)=\R^d$. This implies
the first part of the lemma.

Let $\theta\in  \ci_F$ be such  that $ |\Sigma_\theta| =0$.   Then there
exists $h\in  \R^d\setminus \{0\}$  such that $\langle  h, \Sigma_\theta
h\rangle=0$.    This   implies   that  $\P_\theta$-a.s.    $\langle   h,
X\rangle=c$  with $c=\langle  h, m_\theta\rangle$.   This equality  also
holds $\P$-a.s.   as the two  probability measures $\P$  and $\P_\theta$
are equivalent.   Since $\aff(F)=\R^d$, we get $h=0$. Since this is
absurd, we  deduce that $  |\Sigma_\theta| >0$  for all
$\theta\in  \ci_F$.  Then  use  the continuity of $\theta\mapsto
|\Sigma_\theta|$ on $\ci_F$ to get the second part of the lemma. 
\end{proof}

\subsection{Gnedenko's $d$-dimensional local theorem}
\label{sec:Gnedenko}

Recall  the  definitions  of   $\varphi$,  $\P_\theta$,  $m_\theta$  and
$\Sigma_\theta$   given  by   \reff{eq:def-fi},  \reff{eq:P_theta}   and
\reff{eq:m-s} and  that $\ci_F=\inter \dom (\varphi)$.  The next theorem
is an extension  of the one-dimensional theorem  of Gnedenko \cite{G48},
see also \cite{R61,S66}.

\begin{theo}
\label{Gth2}
Let $F$  be an  aperiodic probability distribution  on $\Z^d$  such that
$\ci_F$   is  non-empty.   Let  $(X_\ell,   \ell\in\mathbb{N}^*)$  be
independent   random   variables   with   distribution   $F$   and   set
$S_n=\sum_{\ell=1}^n  X_\ell$ for  $n\in  \N^*$.  Then  for any  compact
subset    $K$   of $\ci_F$,    we   have:
\begin{equation}
   \label{d2}
\lim_{n\rightarrow\infty}\, \sup_{\theta\in       K}\, \sup_{s\in\mathbb{Z}^d}
\Big|n^{{d}/{2}}|\Sigma_\theta|^{{1}/{2}}\P_\theta(S_n=s)-
(2\pi)^{-d/2}   \expp{-\norm{z_n(\theta,s)}^2/2} \Big|= 0, 
    \end{equation}   
with $z_n(\theta,s)=n^{-1/2}\Sigma_{\theta}^{-1/2} (s-nm_{\theta})$.
\end{theo}

The end of this section is devoted to the proof of Theorem \ref{Gth2}.

Let   $K\subset \ci_F$ be compact. 
Thanks to Lemmas \ref{lem:L3b} and \ref{lem:dound-i}, we have
$|\Sigma_\theta|>0$ and $\Sigma^{-1/2}_\theta$ is well defined.
We define:
\begin{equation}
 \label{defi:Y}
 Y=n^{-1/2}\Sigma_{\theta}^{-1/2} (X_1-m_{\theta})
\quad\text{and}\quad
f_{\theta}(t)=\E_\theta\left[\expp{i\langle t,\, Y\rangle}\right].
\end{equation}
By the inversion formula, we know that for $s\in \Z^d$:
\begin{align*}
(2\pi)^d \P_\theta(S_n=s)
&=\int_{(-\pi,\pi)^d}
\E_\theta\left[\expp{i\langle u , S_n-s \rangle }\right]du\\
&=\int_{(-\pi,\pi)^d}
\E_\theta\left[\expp{i\langle n^{1/2}\Sigma_\theta^{1/2} u, \,
n^{-1/2}\Sigma_\theta^{-1/2} (S_n-s)\rangle}\right]du\\ 
&=\int_{(-\pi,\pi)^d}\E_\theta\left[\expp{i\langle
  n^{1/2}\Sigma_{\theta}^{1/2} u,\, Y\rangle}\right]^n\, 
\expp{-i\langle n^{1/2}\Sigma_{\theta}^{1/2} u ,\,
  z_n(\theta,s)\rangle}du.
\end{align*}
In   order  to   simplify  the   notation,  we   shall  write   $z$  for
$z_n(\theta,s)$.    By    considering    the    change    of    variable
$t=n^{1/2}\Sigma_{\theta}^{1/2} u$, we obtain:
\[
(2\pi)^d\P_{\theta}(S_n=s)
=n^{-d/2}|\Sigma_\theta|^{-1/2} \int_{\mathcal{J}_\theta} 
f_\theta(t) ^n\, \expp{-i\langle
  t,\,  z\rangle}dt,
\]
where                            $\mathcal{J}_\theta=\{t\in\mathbb{R}^d:
n^{-1/2}\Sigma_{\theta}^{-1/2}  t\in(-\pi,\pi)^d\}$.  We
set:
\[
    I_n(\theta)=n^{d/2}|\Sigma_\theta|^{1/2}
\P_{\theta}(S_n=s)-(2\pi)^{-d/2}\expp{-\norm{z}^2/2}.
\]
  Notice that
\[
  (2\pi)^{d/2}\expp{-\norm{z}^2/2}=\int_{\mathbb{R}^d}
\expp{-\norm{t}^2/2-i\langle   t,z\rangle}dt.
\]
We obtain:
\[
(2\pi)^d \, I_n(\theta)
= \int_{\R^d}
\left(\ind_{\mathcal{J}_\theta}(t) f_{\theta}(t)^n- \expp{-\norm{t}^2/2}\right) \expp{-i\langle
  t,z\rangle}dt.
\]

Let $(C_n,n\in \N^*)$ be a sequence of positive numbers such that:
\begin{equation}
   \label{eq:def-Cn}
\lim_{n\rightarrow\infty}C_n=\infty \quad \text{and} \quad
\lim_{n\rightarrow\infty}n^{-1/(12+6d)}C_n=0.
\end{equation}
We  deduce, using the expression  of $\Sigma_\theta^{-1}$ based
on the cofactors, that $\theta\mapsto \Sigma_\theta ^{-1}$ is continuous
on  $\ci_F$.  This implies that  $\norm{\Sigma_{\theta}^{-1/2}t}^2= \langle
t,\Sigma_{\theta}^{-1}t\rangle$ is continuous in $(\theta, t)$ on
$\ci_F\times \R^d$. We
deduce that:
\begin{equation}
\label{defi:c1}
 c_1:=\sup_{\theta\in K,\, \norm{t}=1}\langle t,\Sigma_{\theta}^{-1}
 t\rangle<\infty. 
\end{equation} 
Set  $J_1=\{t\in\mathbb{R}^d; \, 
\norm{t}\leq             C_n\}$, so that $t\in J_1$ implies 
$\norm{n^{-1/2}\Sigma_{\theta}^{-1/2}  t}^2\leq n^{-1}c_1\norm{t}^2 \leq
n^{-1}c_1 C_n^2$. Thanks to \reff{eq:def-Cn}, we get there exists $n_1$ finite,
such that $J_1\subset \mathcal{J}_\theta$ for all $n\geq n_1$ and all $\theta\in K$.  

For
$\varepsilon\in (0,1)$ and $n\geq n_1$, we obtain:
\begin{equation}
   \label{eq:In(theta)}
(2\pi)^d|I_n(\theta)|  
\leq  \int_{\R^d}
\val{\ind_{\mathcal{J}_\theta}(t) f_{\theta}(t)^n- \expp{-\norm{t}^2/2}}
dt
\leq I_{n,1}(\theta)+I_{n,2}(\theta)+I_{n,3}(\theta)+I_{n,4}   ,
\end{equation}
with
\[
I_{n,1}(\theta)=\int_{J_1}|f_{\theta}(t)^n
-\expp{-\norm{t}^2/2}|dt,
\quad 
I_{n,2}(\theta)=\int_{J_{2,\theta}}|f_{\theta}(t)|^n dt,\quad
I_{n,3}(\theta)=\int_{J_{3,\theta}}|f_{\theta}(t)|^ndt,
\]
and $I_{n,4}=\int_{  J_1^c}\expp{-\norm{t}^2/2}dt$ as well as
        $J_{2,\theta}=\{t\in\mathbb{R}^d;
\,\norm{t}>C_n \text{ and } n^{-1/2}\norm{\Sigma_\theta^{-1/2}
t}<\varepsilon\}$,                   $J_{3,\theta}=\{t\in \cj_\theta; \,
n^{-1/2}\norm{\Sigma_\theta^{-1/2}
t}\geq \varepsilon \}$. The proof of the Theorem will be complete as
soon as we prove the converge of the terms $I_{n,i}$ to 0 for
$i\in \{1, \ldots, 4\}$ uniformly for $\theta\in K$ (notice the terms $I_{n,i}$ do
not depend on $s\in \Z^d$).\\

\subsubsection{Convergence of $I_{n,4}$}

Notice that $I_{n,4}$ does not depend on $\theta$. And we deduce from
\reff{eq:def-Cn} that $\lim_{n\rightarrow\infty}I_{n,4}=0$. 

\subsubsection{Convergence of $I_{n,3}$}

Set   $h(\theta,u)=|\E_{\theta}[\expp{i\langle  u,   X_1\rangle}]|$  for
$u\in\mathbb{R}^d$        and       $L=\{u\in        [-2\pi+\varepsilon,
2\pi-\varepsilon]^d;  \,  \norm{u}\geq  \varepsilon\}$.   Since  $F$  is
aperiodic,  we  deduce from  Proposition  P8  in \cite[p.75]{S01},  that
$h(\theta,u)<1$ for $u\in L$.  Since  $h$ is continuous in $(\theta, t)$
on  the compact  set  $K\times  L$, there  exists  $\delta<1$ such  that
$h(\theta,u)\leq \delta$ on $K\times  L$. We get for $\theta\in K$:
\[
I_{n,3}(\theta)
\leq n^{d/2}|\Sigma_\theta|^{1/2}
\int_{(-\pi,      \pi)^d}
h(\theta,u)^n\, \ind_{\{\norm{u}\geq \varepsilon\}}\, du 
\leq                n^{d/2}|\Sigma_\theta|^{1/2}
 (2\pi)^d \delta^n,   
\]
where we used that $|f_\theta(t)|=h(\theta,u)$ with
$t=n^{1/2}\Sigma_{\theta}^{1/2} u$ for the first inequality
and that $h$ is bounded by $\delta$ on $\{u\in 
 (-\pi, \pi)^d;\,  \norm{u}\geq \varepsilon\}$. 
Thanks to \reff{eq:L3bound} we have  $\sup_{\theta\in
  K}|\Sigma_\theta|<\infty$ and since $\delta<1$, we get $
\lim_{n\rightarrow\infty}\sup_{\theta\in
  K}I_{n,3}(\theta)=0$.

\subsubsection{Convergence of $I_{n,2}$}

From \reff{eq:L3bound}, we have
\begin{equation}
\label{defi:a2}
a_2:=\sup_{\theta\in K}\E_\theta[\norm{X_1-m_\theta}^2]<\infty
\quad \text{and}\quad
a_3:=\sup_{\theta\in K}\E_\theta[\norm{X_1-m_\theta}^3]<\infty.
\end{equation}
Using $c_1$ defined in \reff{defi:c1}, we can choose $\varepsilon$ small enough such that
\begin{equation}
\label{condi:epsilon}
 \varepsilon^2 a_2+ \varepsilon a_3 c_1<1.
\end{equation} 

Recall $Y=n^{-1/2}\Sigma_{\theta}^{-1/2} (X_1-m_{\theta})$.
By the symmetry of $\Sigma_{\theta}$, 
we get that
\begin{equation}
   \label{eq:EY2}
\E_{\theta}\left[\norm{Y}^2\right]=
\frac{1}{n} 
\E_{\theta}\Big[\langle
X_1-m_{\theta},\Sigma_{\theta}^{-1}(X_1-m_{\theta})\rangle\Big] 
=\frac{1}{n}\sum_{j=1}^d\sum_{\ell=1}^d
\Big[\Sigma_{\theta}^{-1}(j, \ell)\Sigma_{\theta}(\ell, j)\Big]
=\frac{d}{n}\cdot
\end{equation}
Using similar computations, we obtain:
\begin{equation}
   \label{equ:tY2}
\E_{\theta}\left[\langle t,Y\rangle^2\right]= \frac{\norm{t}^2}{n}\cdot
\end{equation}

 Recall notations $a_3$ in (\ref{defi:a2}) and $c_1$ in (\ref{defi:c1}).
For $t\in J_{2,\theta}$, we get:
\begin{equation}
\label{equ:tY3}
\E_{\theta}\left[|\langle t, Y\rangle|^3\right]
\leq n^{-3/2} \norm{\Sigma_{\theta}^{-1/2} t}^3\, 
\E_{\theta}[\norm{X_1-m_{\theta}}^3]
\leq\frac{\norm{t}^2}{n}\, \varepsilon a_3 c_1\leq \frac{\norm{t}^2}{n},
\end{equation}
where we used
 $n^{-1/2}\norm{\Sigma_{\theta}^{-1/2} t}<\varepsilon$, \reff{defi:a2}
 and \reff{defi:c1}
 for the second inequality and
 \reff{condi:epsilon} for the last. Recall $a_2$ given in
 \reff{defi:a2}. From \reff{condi:epsilon} and since $t\in
 J_{2,\theta}$, we get:
\begin{equation}
   \label{eq:Y<1}
 \E_\theta\left[\langle t, Y\rangle^2\right]
\leq
\norm{n^{-1/2}\Sigma_{\theta}^{-1/2} t}^2\, 
 \E_\theta[\norm{X_1-m_\theta}^2]\leq \varepsilon^2 a_2<1.
 \end{equation} 
 We  deduce that, for all $\theta\in K$ and $t\in J_{2,\theta}$,
\begin{align*}
|f_\theta(t)|=
|\E_{\theta}[\expp{i\langle t,Y\rangle}]|
&=\Big|1-\frac{\E_{\theta}[|\langle t,Y\rangle|^2]}{2}-i
\E_{\theta}\left[\int_0^{\langle t,Y\rangle}\int_0^v\int_0^s \expp{iu}du
  ds dv\right]\Big|\\ 
&\leq 1-\frac{\E_{\theta}[|\langle
  t,Y\rangle|^2]}{2}+\E_{\theta}\left[\int_0^{|\langle
  t,Y\rangle|}\int_0^v\int_0^s ~du ds dv\right]\\ 
&= 1-\frac{\E_{\theta}[|\langle
  t,Y\rangle|^2]}{2}+\frac{\E_{\theta}[|\langle t,Y\rangle|^3]}{6}\\ 
&\leq 1-\frac{\norm{t}^2}{2n}
+\frac{\norm{t}^2}{6n}=1-\frac{\norm{t}^2}{3n},
\end{align*}
where we used that $\E_\theta[Y]=0$
for the first equality, that $\E_\theta[\langle t, Y\rangle^2]\leq 1$
for the first inequality (see \reff{eq:Y<1}) and 
\reff{equ:tY2} as well as \reff{equ:tY3} for the last inequality.
Therefore, we get that:
\[
I_{n,2}(\theta)\leq
\int_{J_{2,\theta}}|f_\theta(t)|^n dt
\leq\int_{J_{2,\theta}}\left(1-\frac{\norm{t}^2}{3n}\right)^n dt
\leq\int_{\norm{t}>C_n}\expp{-\norm{t}^2/{3}}dt.
\]
Since $\lim\limits_{n\rightarrow\infty}C_n=\infty$, we deduce that 
$\lim_{n\rightarrow\infty}\sup_{\theta\in K} I_{n,2}(\theta)=0$.

\subsubsection{Convergence of $I_{n,1}$}

Since $|f_\theta(t)|\leq 1$, we have:
\begin{equation}
   \label{eq:majof1}
|f_\theta(t)^n - \expp{-\norm{t}^2/2}|
\leq  n |f_\theta(t) - \expp{-\norm{t}^2/(2n)}|
\leq n|h_\theta (n,t)| + ng(n,t),
\end{equation}
where
\[
h_\theta (n,t)= f_\theta(t) - 1 + \frac{\norm{t}^2}{2n}
\quad\text{and}\quad
g(n,t)=\Big|\expp{-\norm{t}^2/(2n)} -1 +
\frac{\norm{t}^2}{2n}\Big|.
\]
Since $0\leq  x+ \expp{-x} -1\leq  x^2/2 $ for $x\geq 0$, we get for
$t\in J_1$:
\begin{equation}
   \label{eq:majog1}
ng(n,t)\leq 
\frac{\norm{t}^4}{8n } \leq  n^ {-1} C_n^ 4 . 
\end{equation}

Since $\E_\theta[Y]=0$ and $\E_\theta\left[\langle t,Y \rangle^2\right]=\norm{t}^2/n$,
see \reff{equ:tY2}, we deduce that:
\[
h_\theta (n,t)=\E_\theta\left[\expp{i\langle t, Y \rangle} -1 + i\langle
  t, Y \rangle + \frac{\langle t,
      Y\rangle^2}{2}\right]. 
\]
Let $L_n=n^{\frac{1}{4}}$. 
We have:
\begin{align*}
|h_{\theta}(n,t)|
&\leq
 \E_\theta\left[\Big|\expp{i\langle t, Y\rangle}-1+i\langle t, Y
    \rangle+ \frac{\langle t,
      Y\rangle^2}{2}\Big|\right]\\ 
&=
 \E_\theta\left[\Big|\expp{i\langle t, Y\rangle}-1+i\langle t, Y
    \rangle+ \frac{\langle t,
      Y\rangle^2}{2}\Big|; \, \norm{X_1-m_\theta}<L_n\right]\\ 
& \hspace{2cm}
+\E_\theta\left[\Big|\expp{i\langle t, Y\rangle}-1+i\langle t, Y
    \rangle+ \frac{\langle t,
      Y\rangle^2}{2}\Big|; \, \norm{X_1-m_\theta}\geq L_n\right]\\ 
&\leq\inv{6}\E_\theta\left[|\langle t,
  Y\rangle|^3;\, |X_1-m_\theta|<L_n\right] 
+\E_\theta\left[\langle t, Y\rangle^2; \, \norm{X_1-m_\theta}\geq
  L_n\right], 
\end{align*}
where   we   used  $|\expp{i\alpha}-1-i\alpha+\frac{\alpha^2}{2}|   \leq
\min(|\alpha|^3/6,  \alpha^2)$   for  $\alpha\in  \R$  for   the  second
inequality.  We have:     
\begin{align*}
\E_\theta[|\langle     t,
Y\rangle|^3;\, \norm{X_1-m_\theta}<L_n]
&=\E_\theta\left[\langle     t,
Y\rangle^2 |\langle t,n^{-1/2}\Sigma_{\theta}^{-1/2}(X_1-m_\theta)
\rangle| ; \, \norm{X_1-m_\theta}<L_n\right]\\ 
&\leq n^{-1/2} \norm{t} \sqrt{c_1}\, L_n    
\E_\theta\left[\langle     t,
Y\rangle^2\right]\\
&= n^{-3/2}\norm{t}^3\sqrt{c_1}\, L_n    , 
\end{align*}
 where   we  used  $c_1$ defined in
\reff{defi:c1}    for the inequality and 
\reff{equ:tY2}  for the
last equality. Hölder's inequality gives:
\[
\E_\theta\left[\langle t, Y\rangle^2; \, \norm{X_1-m_\theta}\geq L_n\right]
\leq   \E_\theta\left[|\langle t, Y\rangle|^3\right]^{2/3}
\P_\theta(\norm{X_1-m_\theta}\geq L_n)^{1/3}.
\]
Using $a_3$ defined  in \reff{defi:a2}, we get:
\[
 \E_\theta\left[|\langle t, Y\rangle|^3\right]
\leq  n^{-3/2} \norm{\Sigma_\theta^{-1/2} t}^ 3 \E_\theta\left[\norm{X_1
    -m_\theta}^3 \right]
 \leq  n^{-3/2} c_1^{3/2} \norm{ t}^ 3
a_3.
\]
Using Tchebychev's inequality and $a_2$ defined  in \reff{defi:a2}, we get:
\[
\P_\theta(\norm{X_1-m_\theta}\geq L_n)
\leq  
\E_\theta\left[\norm{X_1-m_\theta}^2\right] L_n^{-2}\leq  a_2 L_n^{-2}.
\]
This gives:
\[
\E_\theta\left[\langle t, Y\rangle^2; \, \norm{X_1-m_\theta}\geq L_n\right]
\leq  n^{-1} c_1 \norm{t}^2 a_3^{2/3} a_2^{1/3}L_n^{-2/3}.
\]
For $t\in J_1$, that is $\norm{t}\leq  C_n$, we get:
\[
n |h_\theta(n,t)|
\leq \inv{6} n^{-1/4} C_n^3 \sqrt{c_1}+ n^{-1/6} c_1 C_n^2 a_3^{2/3} a_2^{1/3}
.
\]
Using \reff{eq:majof1} and \reff{eq:majog1}, we deduce there exists a
constant $c$ which does not depend on $t$, $\theta$ and $n$ such that
for $t\in J_1$, $\theta\in K$, we have:
\[
|f_\theta(t)^n - \expp{-\norm{t}^2/2}|\leq  c (n^{-1/4} C_n^3+n^{-1/6}
 C_n^2+  n^{-1}C_n^4).
\]
We deduce that for $\theta\in K$:
\[
I_{n,1}(\theta)= \int_{J_1} |f_\theta(t)^n - \expp{-\norm{t}^2/2}|
\leq  c (n^{-1/4} C_n^3+n^{-1/6}
 C_n^2+  n^{-1}C_n^4) 2^d C_n^d.
\]
Recall that  $\lim_{n\rightarrow\infty}n^{-1/(12+6d)}C_n=0$.
This implies  $
\lim_{n\rightarrow\infty}\sup_{\theta\in K}I_{n,1}(\theta)=0$.

\subsection{Strong ratio limit theorem}
\label{sec:rlrw}
Recall Definition \ref{defi:aperiodic} for an aperiodic probability distribution.
Consider an aperiodic distribution $F$ on $\Z^d$.
Let $X$ be a random variable with  distribution $F$.
Recall the function $\varphi(\theta)=\log\E[\expp{\langle\theta,
  X\rangle}]$ defined in \reff{eq:def-fi}  and its conjugate $\psi$
defined in \reff{eq:psi}. 
We state the following strong ratio theorem, which is of interest by
itself. However, in this paper we used the extension of the strong
ratio theorem given in Section \ref{sec:ext}.

\begin{theo}
   \label{theo:neveuth}
Let $F$ be an aperiodic probability distribution on $\Z ^d$. Let $(X_\ell, \ell\in \N^*)$ be independent random variables with the same
 distribution $F$. Let $S_n=\sum_{\ell=1}^n X_\ell$ for $n\in
\N^*$. For all $m\in \N$ and $b\in \Z^d$, we have:
\begin{equation}
   \label{eq:neveu1}
\lim_{n\rightarrow\infty}\frac{\P(S_{n-m} =s_n-b)}{\P(S_n=s_n)}=1,
\end{equation}
where the sequence $(s_n, n\in \N^*)$ of elements of $\Z ^ d$ satisfies
the following conditions:
\begin{itemize}
\item[(a)] $\sup_{n\in \N ^*} |\frac{s_n}{n}|<\infty$,
\item[(b)] $\lim_{n\rightarrow\infty}\psi(\frac{s_n}{n})=0$.
\end{itemize}
\end{theo}

\begin{rem}
   \label{rem:E[X]}
   Assume that  $X$, with  distribution $F$,  is integrable.   Thanks to
   Corollary \ref{cor:ri-F},  $\E[X]$ belongs  to $\ri  \dom(\psi)$, the
   relative    interior     of    the     domain    of     $\psi$    and
   $\psi(\E[X])=0$.      According     to      Theorem     1.2.3      in
   \cite{auslender2006asymptotic},  the  function $\psi$  is  relatively
   continuous  on $\ri  \dom(\psi)$.  Therefore  if the  sequence $(s_n,
   n\in    \N^*)$     of    elements    of     $\dom(\psi)$    satisfies
   $\lim_{n\rightarrow\infty}s_n/n=\E[X]$, then  (a) and (b)  of Theorem
   \ref{theo:neveuth}  are  satisfied.  Notice   also  that  if  $F$  is
   aperiodic  (as assumed  in Theorem  \ref{theo:neveuth}), then  Lemmas
   \ref{lem:dound-i} and  \ref{lem:ri-F} imply  $\ri \dom(\psi)$  is the
   (non-empty)  interior   of  $\dom(\psi)$  which  is   also  equal  to
   $\co_F=\inter \cv(F)$.
\end{rem}

\subsection{Proof of Theorem \ref{theo:neveuth}}
We adapt  the proof of Neveu  \cite{N63}. We first state a preliminary
lemma.
\begin{lem}
   \label{lem:s/n=0}
Let $F$ be an aperiodic probability distribution on $\Z ^d$. Let $(s_n,
n\in \N^*)$ be elements of $\Z ^ d$ satisfying (a) and (b) of Theorem
\ref{theo:neveuth}. Then, for all $b\in \Z^d$ and $m\in \Z$, we have 
$\lim_{n\rightarrow\infty}\psi(\frac{s_n+b}{n+m})=0$.
\end{lem}
\begin{proof}
  Assume that (a) and (b) of Theorem \ref{theo:neveuth} hold. Let $x$ be
  a limit  of a converging  sub-sequence of $(s_n/n, n\in  \N^*)$. Since
  $\psi$  is   lower-semicontinuous  and   non-negative,  we  deduce
  from (b)  that
  $\psi(x)=0$.    Thus,  the   possible  limits   of  sub-sequences   of
  $((s_n+b)/(n+m),  \, n+m\geq  1)$,  which are  also  the the  possible
  limits  of sub-sequences  of $(s_n/n,  \,  n\in \N^*)$,  are zeros  of
  $\psi$. Then, using the second part of Lemma \ref{lem:ri-F} and the
  continuity of $\psi$ 
  on     the  interior   of   its   domain,  we   deduce   that
  $\lim_{n\rightarrow\infty}\psi(\frac{s_n+b}{n+m})=0$.
\end{proof}

Since $F$  is aperiodic,  using elementary arithmetic  consideration and
Lemma \ref{lem:s/n=0}, we see it is enough to prove \reff{eq:neveu1} for
$m=1$ and $b\in \Z^d$ satisfying $\mathfrak{p}:=\P(X_1=b) >0$.

We set  $N_n=\Card (\{\ell\leq n;
X_\ell=b\})$. Since for $a\in\Z^d$ the conditional probability
$\P(X_\ell=b|S_n=a)$ does not depend on $\ell$ (when $1\leq \ell\leq
n$), we get:
\[
\E\left[\frac{N_n}{n}\, \Big |\, S_n=a\right]
=\P(X_n=b|S_n=a)
= \mathfrak{p} \frac{\P(S_{n-1} =a-b)}{\P(S_n=a)}\cdot
\]
For $\varepsilon>0$, we have:
\begin{equation}
   \label{eq:majo-ratio}
\Big|\frac{\P(S_{n-1}=a-b)}{\P(S_n=a)}-1\Big|
=\Big|\frac{\E\left[\frac{N_n}{n}; S_n=a\right]}{\mathfrak{p}\P(S_n=a)}-1\Big|
\leq\frac{\E[|\frac{N_n}{n}-\mathfrak{p}|;S_n=a]}{\mathfrak{p}\P(S_n=a)}
\leq \frac{\varepsilon}{\mathfrak{p}}+\frac{R_n(a)}{\mathfrak{p}},
\end{equation}
with 
\[
R_n(a)=
\frac{\P(|\frac{N_n}{n}-\mathfrak{p}|>\varepsilon)}
{\P(S_n=a)}\cdot
\]
Thus, the proof will be complete as soon as we prove that for all
$\varepsilon>0$, $\lim_{n\rightarrow\infty }R_n(s_n)=0$.

By Hoeffding's inequality, see Theorem 1 in \cite{H63}, since
$N_n$ is binomial with parameter
$(n,\mathfrak{p})$, we get:
\begin{equation}
   \label{eq:hoeffding}
\P\left(\val{\frac{N_n}{n}-\mathfrak{p}}>\varepsilon\right)\leq  2
\expp{ -2 n\varepsilon^2}. 
\end{equation}

We give a lower bound of $\P(S_n=s_n)$ in the next lemma, whose proof is
postponed to the end of this section.
\begin{lem}
   \label{lem:lowerbd}
Let $F$ be an aperiodic probability distribution on $\Z ^d$. Let
$(X_\ell, \ell\in \N^*)$ be independent random variables with the
same distribution $F$. Let $S_n=\sum_{\ell=1}^n X_\ell$ for $n\in
\N^*$. Then for $0<\eta<1$, $K_0$ compact subset of $\co_F$, $(s_n, n\in
\N^*)$ a sequence of elements of $\Z^d$ such that $s_n/n\in K_0$, there
exists some $n_0\geq 1$ such that for  $n\geq n_0$ we have:
\[
\P(S_{n} =s_n)\expp{n\psi(s_n/n)}\geq (1-\eta)^n.
\]
\end{lem}

Using
\reff{eq:hoeffding} and Lemma \ref{lem:lowerbd} with
$1-\eta=\expp{-\varepsilon^2}$, we get: 
\[
R_n(s_n)=\frac{\P\left(\val{\frac{N_n}{n}-\mathfrak{p}}
    >\varepsilon\right)}{\P(S_n=s_n)}     
\leq  2 \expp{-n\varepsilon^2 + n\psi(s_n/n)}.
\]
Since $\lim\limits_{n\rightarrow\infty }  \psi(s_n/n)= 0$ by assumption,
we get the result.\qed

\begin{proof}[Proof of Lemma \ref{lem:lowerbd}]
Since $F$ is aperiodic, Lemma \ref{lem:dound-i} implies that $\co_F$ is
non-empty.

We first assume  that the support of $F$ is  bounded.  In particular the
domain  of  $\varphi$ defined  by  \reff{eq:def-fi}  is $\R^d$.   Recall
notation  \reff{eq:P_theta}  as   well  as  $m_\theta=\E_\theta[X]$  and
$\Sigma_\theta=\Cov_{\theta}(X, X)$.  Let $ K_0$ be a
compact  subset of  $\co_F$.  According to  Lemma \ref{lem:full},  there
exists  a  compact set  $K\subset \R^d$  such  that  $K_0\subset  \{m_\theta,
\theta\in K\}$.  According  to Theorem \ref{Gth2}, we have  that for all
$\varepsilon>0$, there exists $n_0$ such that for all $n\geq n_0$:
\[
\sup_{\theta\in K}\sup_{s\in
  \Z^d}\Big|n^{d/2}|\Sigma_\theta|^{1/2}\P_\theta(S_n=s)-
(2\pi)^{-d/2}
\expp{-u_n(\theta,s)}
\Big|<\varepsilon,
\]
with 
\[
u_n(\theta,s)=\frac{\langle{s}-n{m}_\theta,
\Sigma_\theta^{-1}({s}-n{m}_\theta)\rangle}{2n}\cdot
\]
So we get that for all $n\geq n_0$, $\theta\in K$:
\begin{align*}
\P_{\theta}(S_n={s}_n)
& \geq(2\pi n)^{-d/2}|\Sigma_\theta|^{-1/2}
\expp{-u_n(\theta, s_n)}
-n^{-d/2}|\Sigma_\theta|^{-1/2}
\varepsilon\\
&\geq(2\pi n)^{-d/2}\left(\sup_{q\in K} |\Sigma_q|\right)^{-1/2}
\expp{-u_n(\theta, s_n)}
-n^{-d/2}\left(\inf_{q\in K} |\Sigma_q|\right) ^{-1/2}
\varepsilon.
\end{align*}
We deduce that for all $n\geq n_0$:
\[
\sup_{\theta\in K} \P_{\theta}(S_n={s}_n)
\geq (2\pi n)^{-d/2}\left(\sup_{q\in K} |\Sigma_q|\right)^{-1/2}
\expp{-\inf_{\theta\in K} u_n(\theta, s_n)}
-n^{-d/2}\left(\inf_{q\in K} |\Sigma_q|\right) ^{-1/2}
\varepsilon.
\]
Since $s_n/n$  belongs to  $\{m_\theta; \, \theta\in  K\}$, we  get that
$\inf_{\theta\in  K} u_n(\theta, s_n)=0$.   Thanks to  (\ref{eq:L3bound}) and
Lemma \ref{lem:dound-i}, we can also choose $\varepsilon>0$ 
and     $\delta>0$   both  small     enough      so     that     $     (2\pi
)^{-d/2}\left(\sup_{q\in   K}   |\Sigma_q|\right)^{-1/2}
-\left(\inf_{q\in       K}       |\Sigma_q|\right)       ^{-1/2}
\varepsilon>\delta$. Then we deduce that for all $n\geq n_0$:
\[
\sup_{\theta\in \R^d} \P_{\theta}(S_n={s}_n)
\geq
\sup_{\theta\in K} \P_{\theta}(S_n={s}_n)
\geq  n^{-d/2} \delta>0.
\]
Using  \reff{eq:psi}, we get:
\[
\sup_{\theta\in \R ^d} \P_{\theta}(S_n={s}_n)
= \sup_{\theta\in \R ^d} \P(S_n=s_n) \expp{ \langle\theta, s_n \rangle -
  n\varphi(\theta)
}
= \P(S_n=s_n) \expp{n\psi(s_n/n)}.
\]
This gives, for some $\delta>0$, for all $n\geq n_0$:
\begin{equation}
   \label{eq:majo-bdd}
\P(S_n=s_n) \expp{n\psi(s_n/n)} \geq  \delta  n^{-d/2}>0.
\end{equation}
This gives Lemma \ref{lem:lowerbd} when  the support  of $F$ is
bounded.\\

Let $F$ be  a general aperiodic probability distribution  on $\Z^d$, and
$X$  a  random  variable  with  distribution $F$.   Let  $M>0$  so  that
$\delta_M=\P(|X|>M)<1$. Let $X^M$ be distributed as $X$ conditionally on
$\{|X|\leq       M\}$.          Let
$(X_\ell^M, \ell\in \N)$ be  independent random variables distributed as
$X^M$, and set $S_n^M=\sum_{\ell=1}^n X_\ell^M$. We have:
\[
\P(S_n^M={s}_n)
=\frac{\P(S_n={s}_n,\, |X_\ell|\leq M \text{ for } 1\leq \ell\leq
  n)}{\P(|X|\leq M)^n}
\leq
\frac{\P(S_n={s}_n)}{(1-\delta_M)^n}\cdot
\]
Let $F_M$ be the probability distribution of $X^M$ and $\varphi_M$
defined by \reff{eq:def-fi} with $F$ replaced by $F_M$ and $\psi_M$
defined by \reff{eq:psi} with $\varphi$ replaced by $\varphi_M$.
Since $F$ is aperiodic, we get that $F_M$ is aperiodic for $M$ large
enough. We get:
\[
\P(S_n={s}_n)\expp{n\psi(s_n/n)}
\geq  \P(S_n^M={s}_n)\expp{n\psi(s_n/n)} (1-\delta_M)^n
= \P(S_n^M={s}_n)\expp{n\psi_M(s_n/n)} \expp{ n \Delta_M(s_n/n)},
\]
where we define $\Delta_M(s)=\psi(s) - \tilde \psi_M(s)$ and $ \tilde \psi _M(x) =
\sup_{\theta\in   \R   ^d}   \left(  \langle\theta,   x\rangle   -\tilde
  \varphi_M(\theta)   \right)$   with    $\tilde   \varphi_M(\theta)   =
\log\left(\E\left[\expp{\langle\theta,          X\rangle}\ind_{\{|X|\leq
      M\}}\right]\right) $ so that $\tilde \psi_M(x)=\psi_M(x) - \log(1-\delta_M)$.

Notice  that   the  sequence  of  continuous   finite  convex  functions
$(\tilde  \varphi_M,   M\in  \N^*)$  is  non-decreasing   and  converges
point-wise to  the convex function  $\varphi$ (which is  not identically
$+\infty $  as $\varphi(0)=0$).  By definition,  the sequence  of convex
functions   $(\tilde  \psi_M,   M\in   \N^*)$   is  non-increasing   and
$\tilde \psi_M\geq \psi$.  Therefore the sequence converges to a function
say $\tilde \psi$  such that $\tilde \psi\geq \psi$.   Thanks to Theorem
B.3.1.4 in \cite{HL01} or Theorem  II.10.8 of \cite{c:ca}, $\tilde \psi$
is convex  and $(\tilde \psi_M,  M\in \N^*)$ converges to  $\tilde \psi$
uniformly on  any compact  subset of  $\ri \dom(\tilde  \psi)$.  Theorem
E.2.4.4 in \cite{HL01} gives that  the closure of $\tilde \psi$ (defined
in Definition  B.1.2.4 in  \cite{HL01}) is equal  to $\psi$.   Thanks to
Proposition  1.2.5   in  \cite{auslender2006asymptotic},  we   get  that
$\ri  \dom(\tilde  \psi)=\ri  \dom(\psi)$  and   on  this  set  we  have
$\tilde  \psi=\psi$.  Since  $\ri \dom(\psi)=\ri(F)=\co_F$,  see  Lemmas
\ref{lem:ri-F}    and     \ref{lem:dound-i},    this     implies    that
$\lim_{M\rightarrow+\infty } \Delta_M=0$ uniformly on any compact subset
of $\co_F$.

Notice that $\Delta_M\leq 0$. Therefore for any $\gamma>0$, $K_0$
compact subset of $\co_F$, there exists $M_0$ such that for $M\geq M_0$,
$0\geq  \Delta_M\geq -\gamma$ on $K_0$.
We deduce
from \reff{eq:majo-bdd} with $S_n$ and $\psi$ replaced by $S_n^M$ and
$\psi_M$ that for some $\delta>0$ and $\gamma>0$, there exists $n_0\geq
1$ such that for all $n\geq n_0$:
\[
\P(S_n={s}_n)\expp{n\psi(s_n/n)}
\geq \delta  n^{-d/2} \expp{-\gamma n}.
\]
This completes the proof. 
\end{proof}

\subsection{An extension of  Theorem  \ref{theo:neveuth}}
\label{sec:ext}

We shall need the following extension of  Theorem  \ref{theo:neveuth}. 

\begin{lem}
\label{lem:srlt-H'G}
Let $F$  be a probability distribution  on $\N^{d'}$ which is  aperiodic on
$\Z^{d'}$.  Let $(Y_n,  n\in  \N ^*)$  be  independent random  variables
distributed according to  $F$ and set $W_n=\sum_{\ell=1}  ^n Y_\ell$ for
$n\in \N^*$. Assume  that $\E[|Y_1|]<+\infty $. Let $G$ and  $H'$ be two
random variables in $\mathbb{N}$  and $\mathbb{N}^{d'}$ respectively and
independent of $(Y_n, n\in \N ^*)$ such that $\P(G=0)<1$ and a.s. $G\leq
|H'|^{c}$ for  some $c\geq 1$.  Let  $(w_n, n\in \N^*)$ be  a sequence of
$\N^{d'}$  such that  $\lim_{n\rightarrow+\infty }  w_n/n=\E[Y_1]$.  Then
for any given $\ell\in \N$ and $b$ $\in\mathbb{N}^{d'}$, we have:
\begin{equation}
   \label{eq:neveuGH}
\lim_{n\rightarrow\infty}
\frac{\E [G;\, H'+W_{n-\ell}=w_n-b]}{
\E [G;\, H'+W_{n}=w_n]}=1.
\end{equation}
\end{lem}

\begin{proof}
  Since $F$ is aperiodic and  by elementary arithmetic consideration, it
  is enough to  prove \reff{eq:neveuGH} for $\ell=1$  and $b\in \N^{d'}$
  satisfying  $\mathfrak{p}=\P(Y_1=b) >0$.   Let $\varepsilon>0$.  Using
  similar arguments as in \reff{eq:majo-ratio}, we get:
\[
\Big|
\frac{\E[G;\, H'+W_{n-1 }=w_n-b]}{
\E[G;\, H'+W_{n}=w_n]}-1
\Big|
\leq  
 \frac{\varepsilon}{\mathfrak{p}}+\frac{R_n}{\mathfrak{p}},
\]
 and 
\[
R_n=
\frac{\E\left[G;\, |\frac{N_{n}}{n}-\mathfrak{p}|>\varepsilon,\, 
H'+W_{n}=w_n\right]}
{\E[G; H'+W_n=w_n]},
\]
with $N_n=\sum_{\ell=1}^n  \ind_{\{Y_\ell=b\}}$. 
Choose $g\in\N^*$ and $h\in\N^{d'}$ such that $q=\P(G=g, H'=h)>0$.
We have:
\[
R_n\leq\frac{|w_n|^c\,  
\P\left(\Big|\frac{N_{n}}{n}-\mathfrak{p}\Big|>\varepsilon\right)}
{gq\P(W_n=w_n-h)}
\leq\frac{|w_n|^c\, 2\expp{-2n\varepsilon^2}}
{gq\P(W_n=w_n-h)}
,
\]
where we used  $G\leq |H'|^c$ a.s. and that $H'+W_n=w_n$ implies $H'\leq w_n$
for the first inequality, and 
 inequality \reff{eq:hoeffding} in the Appendix for the second. 
Notice that for all $\varepsilon'>0$ we have $|w_n|^c\leq
\exp(\varepsilon' n)$ for $n$ large enough. 

Then use Lemma \ref{lem:lowerbd} and Remark \ref{rem:E[X]} to conclude
that if $\lim_{n\rightarrow+\infty } w_n/n=\E[Y_1]$, then
$\lim_{n\rightarrow+\infty } R_n=0$. Since $\varepsilon>0$ is arbitrary,
we get $\lim_{n\rightarrow+\infty } \Big|
\frac{\E[G;\, H'+W_{n-1 }=w_n-b]}{
\E[G;\, H'+W_{n}=w_n]}-1
\Big|=0$, which gives 
the result. 
\end{proof}

\subsection*{Acknowledgements}
The authors would like to thank Jean-Philippe Chancelier for pointing
out the  references on convex analysis and his valuable advice as well
as the two anonymous referees for their  comments and suggestions.
H. Guo would like to express her gratitude to J.-F. Delmas for his help
during her stay at CERMICS. 

The research has also been supported by the ANR-14-CE25-0014 (ANR GRAAL). 

\bibliographystyle{abbrv}
\bibliography{bibmulti}

\begin{thebibliography}{10}

\bibitem{AD14b}
R.~Abraham and J.-F. Delmas.
\newblock Local limits of conditioned {G}alton-{W}atson trees: the condensation
  case.
\newblock {\em Elec. J. of Probab.}, 19(56):1--29, 2014.

\bibitem{AD14a}
R.~Abraham and J.-F. Delmas.
\newblock Local limits of conditioned {G}alton-{W}atson trees: the infinite
  spine case.
\newblock {\em Elec. J. of Probab.}, 19(2):1--19, 2014.

\bibitem{AN72}
K.~B. Athreya and P.~E. Ney.
\newblock {\em Branching processes}.
\newblock Springer-Verlag, 1972.

\bibitem{auslender2006asymptotic}
A.~Auslender and M.~Teboulle.
\newblock {\em Asymptotic cones and functions in optimization and variational
  inequalities}.
\newblock Springer Science \& Business Media, 2006.

\bibitem{CL13}
L.~Chaumont and R.~Liu.
\newblock Coding multitype forests: application to the law of the total
  population of branching forests.
\newblock {\em Transactions of the American Mathematical Society},
  368:2723--2747, 2016.

\bibitem{dh}
J.-F. Delmas and O.~Hénard.
\newblock A {W}illiams decomposition for spatially dependent superprocesses.
\newblock {\em Elec. J. of Probab.}, 18(37):1--43, 2013.

\bibitem{G48}
B.~V. Gnedenko.
\newblock On a local limit theorem of the theory of probability.
\newblock {\em Uspekhi Mat. Nauk}, 3(3):187--194, 1948.

\bibitem{GK54}
B.~V. Gnedenko and A.~N. Kolmogorov.
\newblock {\em Limit distributions for sums of independent random variables}.
\newblock English translation, Addison-Wesley, Cambridge, Mass, 1954.

\bibitem{h:cgwtmod}
X.~He.
\newblock Conditioning {G}alton-{W}atson trees on large maximal out-degree.
\newblock {\em J. of Theor. Probab.}, 2016.
\newblock To appear.

\bibitem{HL01}
J.-B. Hiriart-Urruty and C.~Lemar{\'e}chal.
\newblock {\em Fundamentals of convex analysis}.
\newblock Springer Science \& Business Media, 2001.

\bibitem{H63}
W.~Hoeffding.
\newblock Probability inequalities for sums of bounded random variables.
\newblock {\em Journal of the American statistical association},
  58(301):13--30, 1963.

\bibitem{j:sgtcgwrac}
S.~Janson.
\newblock Simply generated trees, conditioned {G}alton-{W}atson trees, random
  allocations and condensation.
\newblock {\em Probab. Surv.}, 9:103--252, 2012.

\bibitem{js:cnt}
T.~Jonnson and S.~Stefansson.
\newblock Condensation in nongeneric trees.
\newblock {\em J. Stat. Phys.}, 142:277--313, 2011.

\bibitem{K86}
H.~Kesten.
\newblock Subdiffusive behavior of random walk on a random cluster.
\newblock {\em Ann. de l'Inst. Henri Poincar\'e}, 22:425--487, 1986.

\bibitem{KLPP97}
T.~Kurtz, R.~Lyons, R.~Pemantle, and Y.~Peres.
\newblock A conceptual proof of the {K}esten-{S}tigum theorem for multi-type
  branching processes.
\newblock In {\em Classical and modern branching processes ({M}inneapolis,
  1994)}, volume~84 of {\em IMA Vol. Math. Appl.}, pages 181--185. Springer,
  1997.

\bibitem{glm90}
J.~A. L.-M. Luis G.~Gorostiza.
\newblock The multitype measure branching process.
\newblock {\em Advances in Applied Probability}, 22(1):49--67, 1990.

\bibitem{M08}
G.~Miermont.
\newblock Invariance principles for spatial multitype {G}alton-{W}atson trees.
\newblock {\em Ann. Inst. H. Poincar{\'e} Probab. Statist}, 44:1128--1161,
  2007.

\bibitem{N63}
J.~Neveu.
\newblock Sur le th\'eor\`eme ergodique de {C}hung-{E}rd{\H o}s.
\newblock {\em C. R. Acad. Sci. Paris}, 257:2953--2955, 1963.

\bibitem{P14}
S.~P{\'e}nisson.
\newblock Beyond q-process: Various ways of conditioning the multitype
  {G}alton-{W}atson process.
\newblock {\em ALEA}, 13:223--237, 2016.

\bibitem{R11}
D.~Rizzolo.
\newblock Scaling limits of {M}arkov branching trees and {G}alton--{W}atson
  trees conditioned on the number of vertices with out-degree in a given set.
\newblock {\em Ann. de l'Inst. Henri Poincar\'e}, 51(2):512--532, 2015.

\bibitem{c:ca}
R.~T. Rockafellar.
\newblock {\em Convex analysis}.
\newblock Princeton Landmarks in Mathematics. Princeton University Press, 1997.

\bibitem{R61}
E.~Rvaceva.
\newblock On domains of attraction of multi-dimensional distributions.
\newblock {\em Select. Transl. Math. Statist. and Probability}, 2:183--205,
  1961.

\bibitem{S01}
F.~Spitzer.
\newblock {\em Principles of random walk}.
\newblock Springer Science \& Business Media, 2013.

\bibitem{S14}
R.~Stephenson.
\newblock Local convergence of large critical multi-type {G}alton-{W}atson
  trees and applications to random maps.
\newblock {\em J. of Theor. Probab.}, 2016.
\newblock To appear.

\bibitem{S66}
C.~Stone.
\newblock On local and ratio limit theorems.
\newblock {\em Proc. of the Fifth Berkeley sympos. on mathematical statistics
  and probability. Berkeley and Los Angeles: Univ. California Press}, 2(part
  II):217--224, 1966.

\end{thebibliography}

\end{document}